\documentclass{amsart}

\usepackage{graphicx}
\usepackage{cite}
\usepackage{verbatim}
\usepackage{psfrag}
\usepackage{amssymb}
\usepackage{amsthm}
\usepackage{amsmath}
\usepackage{mathrsfs}
\usepackage{pstricks,pst-node}
\usepackage{longtable}
\usepackage{array}
\usepackage{ragged2e}
\usepackage{hyperref}
\usepackage{enumitem}
\usepackage{crossreftools}
\input xy
\xyoption{all}
\swapnumbers
\newtheorem{theorem}[subsection]{Theorem}

\newtheorem{corollary}[subsection]{Corollary}

\makeatletter
\newcommand{\optionaldesc}[2]{%
  \phantomsection
  #1\protected@edef\@currentlabel{#1}\label{#2}%
}
\makeatother

\theoremstyle{definition}
\newtheorem{definition}[subsection]{Definition}

\newtheorem{example}[subsection]{Example}

\usepackage{tikz}
\usetikzlibrary{calc}
\usetikzlibrary{matrix}
\usetikzlibrary{arrows.meta}

\newcommand{\RR}{\mathbb R}
\newcommand{\Z}{\mathbb{Z}}
\newcommand{\Q}{\mathbb{Q}}
\newcommand{\QQ}{\mathbb{Q}}
\newcommand{\F}{\mathbb{F}}
\newcommand{\bF}{\mathbf{F}}
\newcommand{\n}{\mathbf{n}}
\DeclareMathOperator{\Aut}{Aut}
\newcommand{\sets}[1]{[\![#1]\!]}
\newcommand{\fps}[1]{[\![#1]\!]}
\DeclareMathOperator{\init}{int}
\DeclareMathOperator{\final}{fin}
\newcommand{\ip}[1]{\bigl \langle #1 \bigr \rangle}
\newcommand{\ipm}[2]{\bigl \langle #2 \bigr \rangle_{#1}}
\newcommand{\ipbw}[1]{\ipm{bw}{#1}}

\newcommand{\fullW}{\widetilde{W}}

\newcommand{\fullXi}{\widetilde{\Xi }}

\def\sA{{\mathscr A}}
\def\sB{{\mathscr B}}
\def\sC{{\mathscr C}}
\def\sG{{\mathscr G}}
\def\sL{{\mathscr L}}
\def\sW{{\mathscr W}}
\def\sX{{\mathscr X}}
\def\sS{{\mathscr S}}
\def\sT{{\mathscr T}}
\def\sZ{{\mathscr Z}}

             \def \bF {{\bf F}}

\def\p{\varpi}

\def\F{\mathbb{F}}

\def\R{\mathbb{R}}

\def\Q{\mathbb{Q}}
\def\Z{\mathbb{Z}}

\makeatletter
\@namedef{subjclassname@2020}{\textup{2020} Mathematics Subject Classification}
\makeatother

\begin{document}

\title[Black-white graphs]
{Black-white polynomials of graphs and generating functions}

\author{Kenneth Goodenough}
\address{Naturwissenschaftlich-Technische Fakult{\"a}t, Universit{\"a}t Siegen, Walter-Flex-Stra{\ss}e 3, 57068 Siegen, Germany}

\author{Paul E. Gunnells}
\address{Department of Mathematics and Statistics, University of
Massachusetts Amherst, Amherst, MA 01003, USA}

\keywords{Generating functions, graph polynomials, Feynman diagrams,
quantum information, quantum networks}

\subjclass[2020]{05A15, 05C31, 81T18, 81P40, 81P45}

\thanks{KG acknowledges support from the Alexander von
Humboldt Foundation. PEG thanks the Simons Foundation for support.}

\begin{abstract}
Let $G$ be a graph.  The \emph{black-white polynomial} $W_G(t)$
enumerates colorings of the vertices of $G$ with two colors (black and
white), where the power of $t$ keeps track of how many white vertices
have an even number of black neighbors.  Such polynomials appear in
quantum information theory, where they are used to capture properties
of the entanglement in certain quantum states described by graphs. In
this paper we describe how to use generating functions to compute
these polynomials for various families $\sG$ of graphs.  Our main
results are the following: (i) we describe some constructions under
which $\sG$ leads to a rational generating function; (ii) we use a
matrix model to construct the exponential generating function of the
black-white polynomials of all graphs; and (iii) we generalize a
construction of Wright to build exponential generating functions of
black-white polynomials for graphs of a given loop number.
\end{abstract}

\maketitle

\section{Introduction}\label{s:intro}

\subsection{} Let $G$ be a graph with vertex set $V = V(G)$ and edge
set $E = E (G)$.  A \emph{black-white coloring} of $G$, or simply a
\emph{coloring}, is a function $c\colon V \rightarrow \Z /2\Z$; here
we call $0\in \Z /2\Z$ (respectively $1$) \emph{white}
(resp.~\emph{black}).  Note that a black-white coloring is not a
coloring in the usual sense of graph theory, since we do not require
that adjacent vertices receive distinct colors.  We write $\sC =
\sC(G)$ for the set of colorings of $G$; clearly $|\sC|= 2^{|G|}$.

\subsection{}
Suppose that $G$ has a coloring $c$, and let $v\in V$. We say the
vertex $v$ is \emph{admissible} if either of the following two
conditions hold:
\begin{enumerate}
\item We have $c (v) = 1$.
\item We have $c (v) = 0$, and the cardinality of the finite set
\[
\bigl\{ w \bigm| \text{$c (w)=1$ and $\{v,w\} \in E$} \bigr\}
\]
is $0\bmod 2$.
\end{enumerate}
In other words any black vertex is admissible, but a white vertex is
admissible only if it has an even number of neighboring black vertices.
We denote by $\alpha (c)\subset V$ the subset of
admissible vertices for the coloring $c$.

\begin{definition}\label{def:defofW}
The \emph{black-white polynomial} $W_{G} (t)\in \Z [t]$ of $G$ is
defined by 
\begin{equation}\label{eq:defofW}
W_{G} (t) = \sum_{c\in \sC (G)} \prod_{v\in \alpha (c)} t.
\end{equation}
More generally, if $H\subset G$ is a subgraph equipped with a fixed
coloring $c (H)$, the the \emph{restricted black-white polynomial}
$W_{G,H,c (H)} (t)$ is defined by 
\begin{equation}\label{eq:defofrW}
W_{G,H,c (H)} (t) = \sum_{\substack{c\in \sC (G)\\
c|_{H} = c (H)}} \prod_{v\in \alpha (c)} t.
\end{equation}
We extend the definition of $W$ to the empty graph by putting
$W_{\emptyset} = 1$.
\end{definition}

\subsection{} The original motivation for these polynomials comes from
quantum information theory, in particular in the study of
\emph{stabilizer states}. Stabilizer states can be interpreted as
self-dual codes over the finite field $\F_4$ with respect to the
Hermitian trace inner product \cite{danielsen2006classification}. Such
codes can be represented by a graph, and it was shown in
\cite{miller2023shor} that the weight enumerator polynomial of such a
self-dual code corresponds---up to a simple reparameterization---to
the black-white polynomial of the corresponding graph.

\subsection{} The goal of this paper is to compute the polynomials
$W_{G} (t)$ for various families of graphs using generating functions.
We begin by discussing our first application.

Let $\sG$ be a countable family of graphs.  We assume that $\sG$ is a
disjoint union 
\[
\sG = \coprod_{n\geq 1} \sG_{n}
\]
and that each $\sG_{n}$ is a finite set of finite order graphs.  In
some examples $\sG_{n}$ will consist of a single graph, sometimes of
order $n$, but this is not required.(\footnote{The empty graph is by
convention the unique graph of order $0$, and for convenience we will
extend the notation in graph families by putting $\sG_{0} =
\{\emptyset \}$.})  We can form the ordinary generating function $\Xi
_{\sG}$ of this family by
\[
\Xi _{\sG}  = \sum_{n\geq 1} \sum_{G\in \sG_{n}} W_{G} (t)x^{n}.
\]
Note that $\Xi_{\sG }$ is a power series in $x$ with coefficients that
are polynomials in $t$.

We say that $\sG$ is a \emph{rational family} if this generating
function is a rational function in $x$ (with coefficients polynomials
in $t$).  As an example, let $\sG$ be the family of path graphs 
\[
\sG_{n} = \{ P_{n} \}, \quad n\geq 0.
\]
In particular $P_n$ is the unique tree of order $n$ with all vertices
having degree $1$ or $2$; we define $P_{0}$ to be the empty graph.  Then we have
\[
\Xi _{\sG} (x) = 1 + (t+1)x + (t^{2}+3)x^{2} + (t^{3}+3t+4)x^{3} +
(t^{4}+2t^{2}+8t+5)x^{4}+\dotsb
\]
and one can show (Example \ref{ex:pathgraph}) that this family is rational:
\begin{equation}\label{eq:pgfn}
\Xi_{\sG} (x) = \frac{1+ (2-2t)x^{2}}{1- (t+1)x + (t^{2}-1)x^{3}}.
\end{equation}
An expression like \eqref{eq:pgfn} makes it possible to compute $W_{G}
(t)$ for large $G$ in the family very quickly.  For instance for the
path graph $P_{100}$ we have
\begin{multline*}
W_{P_{100}} (t) = t^{100}
 + 2t^{98}
 + 102t^{97}
 + 202t^{96}
+ \dotsb +  520112162534040819904t^3 \\
 + 39376811249644776236t^2
 + 1945884509713393346t
 + 47086698309271007,
\end{multline*}
a computation that just takes a few milliseconds.  Our first main
result (Theorem \ref{thm:ratfamilies}) gives several examples of
constructions leading to rational families.  Many related graph
families have been studied in quantum information
theory~\cite{miller2023shor, eltschka2020maximum,
cao2024quantum}.

\subsection{} Our next results concern exponential generating
functions built from the polynomials $W_{G} (t)$ for various families
$\sG$.  As is well known, such generating functions correspond to
counting labeled objects instead of unlabeled objects, or equivalently
to counting unlabeled objects weighted by the inverses of the orders
of their automorphism groups (cf.~\cite{fs}).  More precisely, we
consider generating functions of the form
\[
\Lambda_{\sG} (x) = \sum_{n} \sum_{G\in \sG_{n}} \frac{W_{G}
(t)}{|\Aut G|}x^{n}.
\]
Note that the coefficients of this series lie in $\Q [t]$, not $\Z[t]$.  More generally, we will also
allow families $\sG$ with $|\sG_{n}|$ \emph{infinite} by introducing
additional variables (sometimes countably many) to mark various
features of the graphs in $\sG$.  These generating functions will be
infinite series in $x$ with coefficients elements of the rational
polynomial ring in several (sometimes countably many) variables.

Typically for such examples these generating functions are not
rational.  Indeed in most cases there is no hope of getting a simple
expression for them; one simply accepts them for what they are.  Nevertheless constructing them by various means is important for numerical investigations or for trying to build examples.

\subsection{} There are two main applications we give.  For both of
them, we allow graphs to have loops and parallel sets of edges.  This
entails slight (but obvious) modifications to the definition of $W_{G}
(t)$.  It is essential that we allow
this generality because of the way the methods work.

In the first application, we let $\sG$ be the set of \emph{all}
connected graphs up to isomorphism, and let $\sG_{n}$ be the subset of
graphs of order $n$. Note that since we allow parallel edges and loops
the subset $\sG_{n}$ is infinite.  Therefore we introduce additional
indexing data $\n $ (graph profiles) so that the corresponding subsets
$\sG_{c}(\n)$ are finite.  We then explain in Theorem \ref{thm:genfunallW} how to compute the generating
function
\[
\sB (u) \sum_{\n } u^{|\n |} \sum_{G \in \sG_c (\n
)}\frac{W_{G} (t)\prod_{k} \xi_{k}^{n_{k}}}{| \Aut G|},
\]
where the $\xi_{k}$ are additional marking variables keeping track of
the profile $\n$.  The basic tool is the technique of
Feynman diagrams from perturbative quantum field theory.  The method is
quite flexible, and allows one to build series enumerating graphs with various
restrictions, such as limitations on vertex degrees and topological
complexity. 

\subsection{}
In the second application, we generalize a method of Wright, who built
generating functions for graphs of fixed \emph{loop number} in terms
of certain tree functions \cite{wright1,dyz,jklp}; the results can be found in Theorems \ref{thm:loopnumber0}, \ref{thm:loopnumber1}, and \ref{thm:loopnumber2}. By definition, the loop number $g$ of a
connected graph is its first Betti number (in other words the rank of
its first homology group $H_{1} (G; \Z)$; one can also compute it
using the Euler characteristic as $g=1-|V|+|E|$).  Thus graphs of loop
number $0$ are trees, those of loop number $1$ are cycles with trees
grafted onto the vertices, and so on.  As part of proving these results, we introduce a generalization of the black-white polynomial we call the \emph{full $W$-polynomial}, and show that it can be computed in rational families (Theorem \ref{thm:rffull}).

\subsection{} We finish this introduction with a few remarks.
First, we define various rational families of graphs, but our list is
incomplete.  In fact we don't know a classification of all possible
examples.  We only give a few that seem interesting to illustrate the
method. For more work in this direction, we refer to \cite{cspaper}.

Second, although all the generating functions we consider are really
formal power series, one of the main reasons one introduces generating
functions is to study them analytically to obtain information about
their coefficients.  In particular one can hope to understand the
distribution of the coefficients of the polynomials $W_{G} (t)$ in a
given family. For example, we have observed that for generic families
of graphs, $W_G(t)$ converges to a binomial distribution
$\textrm{binom}(n, 1/4)$ after normalization. Understanding the
asymptotic behavior of the $W_G(t)$ leads to an understanding of the
entanglement of the quantum state associated to $G$, see
\cite{cspaper} for further background.  We will leave such
considerations to a future paper.

Third, there are variations of the definition of $W_{G} (t)$ in which
the finite field $\Z/{2}\Z$ is replaced with the ring $\Z /n\Z$.  For
example see \cite{miller2023shor, bahramgiri2007efficient}.  All our
results can be adapted to this setting, but we leave this
generalization to the reader.

\subsection{Acknowledgments} We thank Elo\"ic Vall\'ee for helpful
conversations.

\section{Rational families}\label{s:rationalfamilies}

\subsection{}
As in 
\S \ref{s:intro} let $\sG = \{\sG_{n} \}_{n\geq 1}$ be a family of graphs with
corresponding generating function
\[
\Xi _{\sG} (x) = \sum_{n\geq 1} \sum_{G\in \sG_{n}} W_{G} (t)x^{n}.
\]
Recall that $\sG$ is a \emph{rational family} if this generating
function is a rational function in $x$ (with coefficients polynomials
in $t$).  The main result of this section is Theorem
\ref{thm:ratfamilies}, which proves that certain graph families are
rational.  We begin with some notation.

First we recall that $P_{n}$ is the path graph on $n$ vertices. We let
$C_{n}$ denote the cycle graph on $n$ vertices.  In both cases we
identify the vertices with $\sets{n} = \{1,\dots ,n \}$. If $H$ is a
subgraph of $G$, we write $i_{H}\colon H \rightarrow G$ for the
inclusion.  For graphs $G$, $G'$, the \emph{product} $G \times G'$ is
the graph with vertex set $V (G)\times V (G')$ and with $(u,u')$
adjacent to $(v,v')$ if either (i) $u$ is adjacent to $v$ in $G$ or
(ii) $u'$ is adjacent to $v'$ in $G'$(\footnote{This is sometimes
called the Cartesian product or the box product.  We remark that it is
not the same as the product in the categorical sense (i.e., it does
not satisfy the universal property of a product).}). For example
$P_{2}\times P_{2} \simeq C_{4}$.

\subsection{}
We now define some families of graphs.  In each case, we fix a graph
$G$ and indicate the
set $\sG_{n}$. 

\begin{enumerate}
\item [\optionaldesc{(GC)}{it:cylinder}] The \emph{$G$-cylinder of length $n$} is the product $G \times
P_{n}$. We let $\sG_{n} = \{G\times P_{n} \}$.
\item [\optionaldesc{(GX)}{it:extrusion}] More generally let $i_{H} \colon H\rightarrow G$ be the
inclusion of $H$ as a subgraph.  Then the \emph{length $n$ extrusion
of $G$ along $H$} is the quotient of 
\[
X_{n} (G,H) = G \coprod (H
\times P_{n}),
\]
where we identify $i_{H} (H)\subset G$ with $H \times
{1}$.  Note that \ref{it:cylinder} is the special case $G = H$.  We
put $\sG_{n} = \{X_{n} (G,H) \}$.
\item [\optionaldesc{(GT)}{it:torus}] The \emph{$G$-torus of length $n$} is the product $G \times
C_{n}$. We put $\sG_{n} = \{G\times C_{n} \}$.
\item [\optionaldesc{(GE)}{it:earring}] More generally if $H$ is a subgraph of $G$, the \emph{length
$n$ earring of $G$ along $H$} is the quotient $E_{n} (G,H) = G \coprod
H \times C_{n}$ where we identify $i_{H} (H)\subset G$ with $H \times
{1}$.  Again \ref{it:torus} corresponds to $G = H$, and we put
$\sG_{n} = \{E_{n} (G,H) \}$. 
\item [\optionaldesc{(GS1)}{it:oneedge}] Let $G$ be a fixed graph and let $e\in E (G)$ be a fixed edge.
For any $n$ let $S_{n} (G,e)$ be the subdivision of $G$ obtained by
adding $n$ new vertices to $e$.  We put $\sG_{n}  = \{S_{n} (G,e)
\}$. 
\item [\optionaldesc{(GSm)}{it:lastone}] More generally let $F = \{e_{1},\dotsc ,e_{m} \}\subset E (G)$ be a fixed
collection of edges.  We let $\sG_{n}$ be the set of graphs obtained
by arbitrarily subdividing the edges in $F$ by adding $n$ new vertices
in total. In other words we put 
\begin{equation}\label{eq:subdividefamily}
\sG_{n} = \bigcup_{n_{1},\dotsc ,n_{m}} \bigcup_{k=1}^{m} S_{n_{k}}( G,e_{k}),
\end{equation}
where the first union is taken over all $n_{k}\geq 0$ with
$\sum {n_{k}} = n$.
\end{enumerate}

\begin{figure}
\centering
\begin{tikzpicture}[
    v/.style={circle, fill=black, draw=black, line width=0.7pt, inner sep=1.6pt},
    eG/.style={draw=black, line width=0.8pt},
    eP/.style={draw=blue!80!black!80, line width=1.0pt, draw opacity=0.65},
    lab/.style={font=\small}
]
\pgfdeclarelayer{bg}
\pgfdeclarelayer{fg}
\pgfsetlayers{bg,main,fg}

\def\r{0.65}
\def\rot{15}
\def\aA{90}
\def\aB{210}
\def\aC{330}

\begin{scope}[shift={(0,0)}]
  \begin{scope}[rotate=\rot]
    \coordinate (n1-1) at (\aA:\r);
    \coordinate (n1-2) at (\aB:\r);
    \coordinate (n1-3) at (\aC:\r);

    \begin{pgfonlayer}{bg}
      \draw[eG] (n1-1) -- (n1-2);
      \draw[eG] (n1-2) -- (n1-3);
    \end{pgfonlayer}

    \begin{pgfonlayer}{fg}
      \foreach \p in {n1-1,n1-2,n1-3}{\node[v] at (\p) {};}
    \end{pgfonlayer}
  \end{scope}
\end{scope}
\node[lab] at (0,-1.25) {$G \simeq G\times P_{1}$};

\begin{scope}[shift={(3.4,0)}]
  \def\dx{1.7}

  \begin{scope}[rotate=\rot]
    \coordinate (n2-1-1) at (\aA:\r);
    \coordinate (n2-2-1) at (\aB:\r);
    \coordinate (n2-3-1) at (\aC:\r);
  \end{scope}

  \begin{scope}[shift={(\dx,0)}, rotate=\rot]
    \coordinate (n2-1-2) at (\aA:\r);
    \coordinate (n2-2-2) at (\aB:\r);
    \coordinate (n2-3-2) at (\aC:\r);
  \end{scope}

  \begin{pgfonlayer}{bg}
    \draw[eG] (n2-1-1) -- (n2-2-1);
    \draw[eG] (n2-2-1) -- (n2-3-1);
    \draw[eG] (n2-1-2) -- (n2-2-2);
    \draw[eG] (n2-2-2) -- (n2-3-2);

    \draw[eP] (n2-1-1) -- (n2-1-2);
    \draw[eP] (n2-2-1) -- (n2-2-2);
    \draw[eP] (n2-3-1) -- (n2-3-2);
  \end{pgfonlayer}

  \begin{pgfonlayer}{fg}
    \foreach \p in {n2-1-1,n2-2-1,n2-3-1,n2-1-2,n2-2-2,n2-3-2}{
      \node[v] at (\p) {};
    }
  \end{pgfonlayer}
\end{scope}
\node[lab] at (3.4+0.85,-1.25) {$G\times P_{2}$};

\begin{scope}[shift={(7.5,0)}]
  \def\dx{1.55}

  \begin{scope}[rotate=\rot]
    \coordinate (n3-1-1) at (\aA:\r);
    \coordinate (n3-2-1) at (\aB:\r);
    \coordinate (n3-3-1) at (\aC:\r);
  \end{scope}

  \begin{scope}[shift={(\dx,0)}, rotate=\rot]
    \coordinate (n3-1-2) at (\aA:\r);
    \coordinate (n3-2-2) at (\aB:\r);
    \coordinate (n3-3-2) at (\aC:\r);
  \end{scope}

  \begin{scope}[shift={(2*\dx,0)}, rotate=\rot]
    \coordinate (n3-1-3) at (\aA:\r);
    \coordinate (n3-2-3) at (\aB:\r);
    \coordinate (n3-3-3) at (\aC:\r);
  \end{scope}

  \begin{pgfonlayer}{bg}
    \foreach \k in {1,2,3}{
      \draw[eG] (n3-1-\k) -- (n3-2-\k);
      \draw[eG] (n3-2-\k) -- (n3-3-\k);
    }

    \foreach \i in {1,2,3}{
      \draw[eP] (n3-\i-1) -- (n3-\i-2);
      \draw[eP] (n3-\i-2) -- (n3-\i-3);
    }
  \end{pgfonlayer}

  \begin{pgfonlayer}{fg}
    \foreach \p in {n3-1-1,n3-2-1,n3-3-1,n3-1-2,n3-2-2,n3-3-2,n3-1-3,n3-2-3,n3-3-3}{
      \node[v] at (\p) {};
    }
  \end{pgfonlayer}
\end{scope}
\node[lab] at (9.05,-1.25) {$G\times P_{3}$}; 

\end{tikzpicture}
\vspace{1mm}
\hrulefill
\vspace*{5mm}
\begin{tikzpicture}[
    v/.style={circle, fill=black, draw=black, line width=0.7pt, inner sep=1.6pt},
    eG/.style={draw=black, line width=0.8pt},
    eP/.style={draw=blue!90!black!50, line width=1.0pt, draw opacity=0.85},
    lab/.style={font=\small}
]
\pgfdeclarelayer{bg}
\pgfdeclarelayer{fg}
\pgfsetlayers{bg,main,fg}

\def\r{0.75}
\def\rot{15}

\def\aA{45}
\def\aB{135}
\def\aC{225}
\def\aD{315}

\def\sOne{0}
\def\sTwo{3.7}
\def\sThree{7.4}

\begin{scope}[shift={(\sOne,0)}]
  \begin{scope}[rotate=\rot]
    \coordinate (A1) at (\aA:\r);
    \coordinate (B1) at (\aB:\r);
    \coordinate (C1) at (\aC:\r);
    \coordinate (D1) at (\aD:\r);

    \begin{pgfonlayer}{bg}
      \draw[eG] (B1) -- (C1);
      \draw[eG] (C1) -- (D1);
      \draw[eG] (D1) -- (A1);
      \draw[eP] (A1) -- (B1);
    \end{pgfonlayer}

    \begin{pgfonlayer}{fg}
      \foreach \p in {A1,B1,C1,D1}{\node[v] at (\p) {};}
    \end{pgfonlayer}
  \end{scope}
\end{scope}
\node[lab] at (\sOne,-1.35) {$G\simeq S_0(G, e)$};

\begin{scope}[shift={(\sTwo,0)}]
  \begin{scope}[rotate=\rot]
    \coordinate (A2) at (\aA:\r);
    \coordinate (B2) at (\aB:\r);
    \coordinate (C2) at (\aC:\r);
    \coordinate (D2) at (\aD:\r);

    \coordinate (E2) at ($(A2)!0.5!(B2)$);

    \begin{pgfonlayer}{bg}
      \draw[eG] (B2) -- (C2);
      \draw[eG] (C2) -- (D2);
      \draw[eG] (D2) -- (A2);
      \draw[eP] (A2) -- (E2);
      \draw[eP] (E2) -- (B2);
    \end{pgfonlayer}

    \begin{pgfonlayer}{fg}
      \foreach \p in {A2,B2,C2,D2,E2}{\node[v] at (\p) {};}
    \end{pgfonlayer}
  \end{scope}
\end{scope}
\node[lab] at (\sTwo,-1.35) {$S_1(G, e)$};

\begin{scope}[shift={(\sThree,0)}]
  \begin{scope}[rotate=\rot]
    \coordinate (A3) at (\aA:\r);
    \coordinate (B3) at (\aB:\r);
    \coordinate (C3) at (\aC:\r);
    \coordinate (D3) at (\aD:\r);

    \coordinate (E3) at ($(A3)!0.333!(B3)$);
    \coordinate (F3) at ($(A3)!0.666!(B3)$);

    \begin{pgfonlayer}{bg}
      \draw[eG] (B3) -- (C3);
      \draw[eG] (C3) -- (D3);
      \draw[eG] (D3) -- (A3);
      \draw[eP] (A3) -- (E3);
      \draw[eP] (E3) -- (F3);
      \draw[eP] (F3) -- (B3);
    \end{pgfonlayer}

    \begin{pgfonlayer}{fg}
      \foreach \p in {A3,B3,C3,D3,E3,F3}{\node[v] at (\p) {};}
    \end{pgfonlayer}
  \end{scope}
\end{scope}
\node[lab] at (\sThree,-1.35) {$S_2(G, e)$};

\end{tikzpicture}
\caption{Examples of $G$-cylinders and subdivisions $S_n(G, e)$.}
\end{figure}
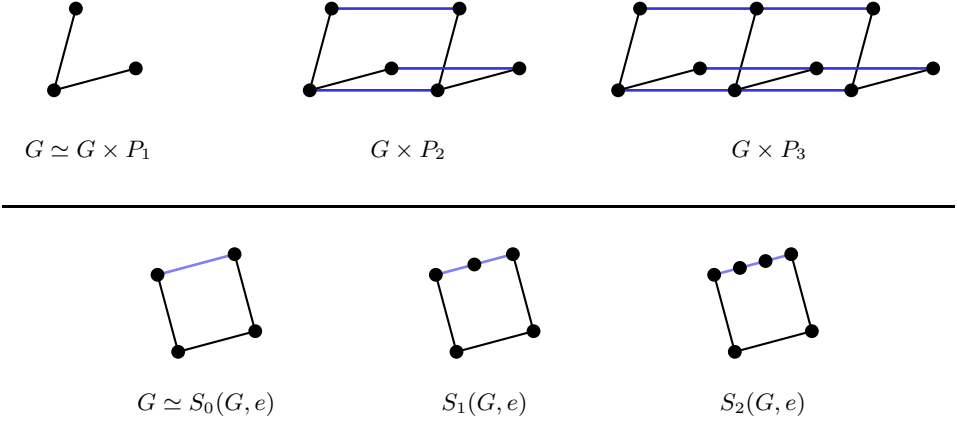

\begin{theorem}\label{thm:ratfamilies}
The graph families \ref{it:cylinder}--\ref{it:lastone} are rational.
\end{theorem}

\subsection{} To prove Theorem \ref{thm:ratfamilies}, we use the
\emph{transfer matrix method}, which we now recall for the convenience
of the reader.  For details see \cite[\S4.7]{stanley}.

A digraph $D$ is the data of a triple $(V, E, \varphi)$ where $V$ is a
finite set of vertices, $E$ is a finite set of directed arcs, and
$\varphi$ is a function $\varphi \colon E \rightarrow V\times V$ that
takes an arc $e$ to the pair $(\init e, \final e)$ of its initial and
final vertices. A \emph{walk} on $D$ of length $N$ is a sequence
$\Gamma$ of edges
$e_{1},\dotsc ,e_{N}$ with $\final e_{i} = \init e_{i+1}$ for $i=1,
\dotsc , N-1$.  We assume that $D$ is equipped with a weight function
$\omega \colon E \rightarrow R$, where $R$ is a commutative ring.  The
\emph{weight of a walk} $\omega (\Gamma)$ is the product $\prod \omega (e)$ of the
weights of each edge in the walk.

Suppose there are $p$ vertices, and identify the set $V$ with the
finite set $\sets{p}$.  Given two vertices $i, j$, we define 
\[
A_{ij} (N) = \sum_{\Gamma} \omega (\Gamma),
\]
where the sum is taken over all walks $\Gamma$ of length $N$ beginning at
vertex $i$ and terminating at vertex $j$.  Since $D$ has finitely many
vertices and edges, this is a finite sum in $R$.  We let $A$ be the
$p\times p$ matrix with $(i,j)$th entry $A_{ij} (1)$.  The matrix $A$
is called the \emph{weighted adjacency matrix} of $D$ with respect to
the weight function $\omega$.   It is a basic result \cite[Theorem
4.7.1]{stanley} that 
\[
A_{ij} (N) = (A^{N})_{ij},
\]
where we take the convention that $A^{0} = I$ even if $A$ is not
invertible.

We collect the $A_{ij} (N)$ into generating functions
\begin{equation}\label{eq:gf1}
F_{ij} (D) (x) = \sum_{N \geq 0} A_{ij} (N)x^{N} = \sum_{N \geq 0}
(A^{N})_{ij} x^{N}, \quad \text{$1\leq i,j\leq p$}.
\end{equation}
We will also need a variation on this theme.  One can pick an initial vector
$v_{0}\in R^{p}$ and define the generating functions
\begin{equation}\label{eq:gf2}
M_{i} (D, v_{0}) (x) = \sum_{N \geq 0} (A^{N}v_0)_{i} x^{N}, \quad \text{$1\leq i\leq p$.}
\end{equation}
The main theorem about the generating functions \eqref{eq:gf1}--\eqref{eq:gf2} is the following:

\begin{theorem}\label{thm:tmm}
\leavevmode
\begin{enumerate}
\item The generating function $F_{ij} (D) (x)$ is rational and is given by
\[
F_{ij} (D) (x) = \frac{(-1)^{i+j} \det (I - xA \colon i,j)}{\det (I-xA)}.
\]
Here $I$ is the $p\times p$ identity matrix, and the notation $(B
\colon i,j)$ means the matrix obtained from $B$ by deleting the $j$th
row and the $i$th column. 
\item The generating function $M_{i} (D, v_{0})(x)$ is rational with
denominator given by $\det (I - xA)$. 
\end{enumerate}
\end{theorem}

\begin{proof}
For a proof of the first statement, see \cite[Theorem 4.7.2]{stanley}.
The second statement is an easy consequence of the first.  Indeed if
we take $v_{0}$ to be the $j$th standard basis vector, then $Av_{0}$
is the $j$th column of $A$ and so $M_{i}=F_{ij}$.  Thus for general
$v_{0}$ the generating function is an $R$-linear combination of the
$F_{ij}$, which implies the result.  
%
\end{proof}
 
\subsection{} Our strategy to prove Theorem \ref{thm:ratfamilies} will be to find for each family
$\sG$ a weighted digraph $D = D (\sG)$.  The weights of the edges $E$
will take values in the ring $R =
\Z [t, t^{-1}]$, the Laurent polynomials in $t$.   We will also
construct a vector $v_{0}\in \Z [t]^{|V|}$, where $V$ is the
vertex set of $D$, such that $\Xi_{\sG} (x)$ is, up to finitely many
terms, a linear combination over $\Z [x]$ of series of the form
\eqref{eq:gf2}.  By Theorem \ref{thm:tmm}, this shows that $\Xi_{\sG} (x)$
is a rational function in $x$ with coefficients in $R$ with denominator $\det (I - xA)$, where
$A$ is the weighted adjacency matrix of $D$.  After multiplying top and bottom by a suitable power of $t$, we see that $\Xi_{\sG} (x)$ is a rational function in $x$ with coefficients in $\Z[t]$.

\begin{proof}
We consider each family in turn, beginning with the $G$-cylinders \ref{it:cylinder}.
Thus each set $\sG_{n}$ consists of the single graph $G\times
P_{n}$, $n\geq 1$.  

Let $\sS $ be the set of all $G$-cylinders of length $2$ equipped with
all possible colorings, and similarly let $\tilde{\sS}$ be the set of
$G$-cylinders of length $3$ equipped with all possible colorings.  Thus
$|\sS | = 4^{|G|}$ and $|\tilde{\sS}| = 8^{|G|}$.  Each $\tilde{S}\in
\tilde{\sS}$ determines two colored cylinders $S, S'\in \sS$ by taking
the subgraphs $G \times \{1,2 \}$ (respectively $G\times \{2,3 \}$) of
$\tilde{S}$ (see Figure \ref{fig:backfront}).  We write $S =
\tilde{S}_{b}$ and $S' = \tilde{S}_{f}$ ($b$ and $f$ stand
respectively for \emph{back} and \emph{front}).  Any $S \in \sS$ or
$\tilde{\sS} $ determines a monomial $\mu_{S} (t) \in \Z [t]$ by
\begin{equation}\label{eq:monomials}
\mu_{S} (t) = \prod_{{v \in \alpha (c)}} t,
\end{equation}
where $c$ is the coloring of $S$.

\begin{figure}
\centering
\resizebox{0.5\textwidth}{!}{
\begin{tikzpicture}

\def\a{1.10}   
\def\b{2.20}   
\def\dx{3.40}  

\newcommand{\cylinder}[3]{%
  \begin{scope}[shift={(#1,#2)}]
    \pgfmathtruncatemacro{\kmone}{#3-1}
    \pgfmathsetmacro{\L}{\dx*\kmone}

    \draw[line width=1.1pt] (0,\b) -- (\L,\b);
    \draw[line width=1.1pt] (0,-\b) -- (\L,-\b);

    \foreach \i in {0,...,\kmone} {
      \pgfmathsetmacro{\x}{\i*\dx}
      \draw[line width=1.2pt] (\x,0) ellipse [x radius=\a, y radius=\b];
      \node[scale=2.1] at (\x,0) {$G$};
    }
  \end{scope}%
}

\cylinder{-1.4}{0}{3}
\node[scale=2.3] at ({1.96}, {\b+0.7}) {$\tilde{S}$};

\node[scale=1.9] at (-1.4,   {-\b-0.5}) {$1$};
\node[scale=1.9] at (-1.4+\dx, {-\b-0.5}) {$2$};
\node[scale=1.9] at ({-1.4+2*\dx}, {-\b-0.5}) {$3$};

\draw[line width=1.0pt,  -{Latex[length=3mm]}] ({-1.4},{-3.2}) -- (-3.0,-5);
\draw[line width=1.0pt,  -{Latex[length=3mm]}] ({5.4},{-3.2}) -- ( 6.8,-5);

\cylinder{-4.6}{-7.6}{2}
\node[scale=2.3] at (-3.0, {-7.6-\b-1.2}) {$\tilde{S}_b$};
\node[scale=1.9] at (-4.6,      {-7.6-\b-0.6}) {$1$};
\node[scale=1.9] at (-4.6+\dx,  {-7.6-\b-0.6}) {$2$};

\cylinder{5.3}{-7.6}{2}
\node[scale=2.3] at (7.0, {-7.6-\b-1.2}) {$\tilde{S}_f$};
\node[scale=1.9] at (5.3,      {-7.6-\b-0.6}) {$2$};
\node[scale=1.9] at (5.3+\dx,  {-7.6-\b-0.6}) {$3$};

\end{tikzpicture}
}
\caption{\label{fig:backfront}  The back and front of a state.}
\end{figure}
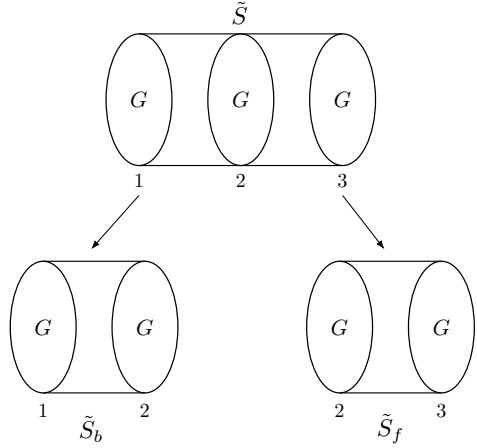

Now we define the weighted digraph $D (\sG)$ as follows:
\begin{enumerate}
\item The vertices of $D$ are in bijection with $\sS$.  We write
$v_{S}$ for the vertex corresponding to $S\in \sS$.
\item If $S, S'\in \sS$, then there is an edge from $v_{S}$ to
$v_{S'}$ if and only if there is an $\tilde{S}\in \tilde{\sS}$ such
that $S = \tilde{S}_{b}$ and $S' = \tilde{S}_{f}$.  
\item If there is an edge $e$ from $v_{S}$ to $v_{S'}$, then $\omega
(e) = \mu_{\tilde{S}} (t)/\mu_{S}(t)$.
\end{enumerate}
We let $A = A (\sG)$ be the weighted adjacency matrix of $D$; it is
$p\times p$ where $p=4^{|G|}$.  We define the initial vector $v_{0}$
to be the $p$-vector whose entry indexed by $S\in \sS$ is the monomial
$\mu_{S} (t)$.

Given all this data, we claim 
\begin{equation}\label{eq:xicylin}
\Xi_{\sG} (x) = 1 + W_{G} (t)x + x^{2} \sum_{i=1}^{p}
M_{i} (D, v_{0}).
\end{equation}
This implies that $G$-cylinders form a rational family.

To prove \eqref{eq:xicylin}, first we observe that the coefficient of
$x^{2}$ is by construction the polynomial $W_{G \times P_{2}} (t)$, since the
elements in $\sS$ range over all colorings of $G \times P_{2}$ and the entries
of $v_{0}$ are just the different monomial contributions of these
colorings.  

Now suppose the result is true for the coefficient of $x^{k}$, namely
that
\[
\sum_{i=1}^{p} (A^{k-2}v_{0})_{i} = W_{G \times P_{k}} (t).
\]
We also suppose that for $S\in \sS$, the entry of $A^{k-2}v_{0}$
indexed by $S$ gives the total contribution to $W_{G \times P_{k}} (t)$ of the
colorings of $G \times P_{k}$ such that the subgraph $G\times \{k-1,k\}$ has
the same coloring as $S$; this is true in the base case $k=2$.

Now consider passing to $A^{k-1}v_{0}$.  By construction the entry of
$A$ indexed by $S$ and some other $S'\in \sS$ describes the change in
this total contribution as $S$ is extended to a colored $3$-cylinder
$\tilde{S}\in \tilde{\sS} $ such that $S=\tilde{S}_{b}$ and
$S'=\tilde{S}_{f}$.  Multiplying by $A$ means that we sum over all
these contributions.  The resulting vector $A^{k-1}v_{0}$ now has
entry $S$ giving the total contribution to $W_{G\times P_{k+1}} (t)$ of all
the colorings of $G \times P_{k+1}$ such that the subgraph $G\times \{k,k+1\}$
is colored the same as $S$.  From this it follows that
\[
\sum_{i=1}^{p} (A^{k-1}v_{0})_{i} = W_{G \times P_{k+1}} (t),
\]
which completes the proof of \ref{it:cylinder}.

We remark that the proof shows why we must take $S,
S'$ to have underlying graphs $G\times P_{2}$:  attaching a copy
of $G$ to pass from $G \times P_{k}$ to $G\times P_{k+1}$ potentially affects the
contributions of the vertices in $G\times \{k \}$ and $G\times \{k-1
\}$, but cannot affect the contributions of
the vertices in the subgraphs $G\times \{q \}$ for $q < k-1$.

Next we consider the extrusion of $G$ along $H$; the proof of
\ref{it:extrusion} is very similar to that of the $G$-cylinder.  The
set $\sS$ (respectively, $\tilde{\sS}$) is the set of colorings of
$H\times P_{2}$ (resp., $H\times P_{3}$), and the weighted digraph $D$
and matrix $A$
are defined in the same way.  The initial vector $v_{0}$ has $S$ entry
equal to the total contribution of all colorings $c$ of $X_{2} (G,H)$
such that the restriction of $c$ to $H\times P_{2}$ agrees with $S$.
The definition \eqref{eq:xicylin} is then the same.   This proves that the extrusion of $G$ along $H$ is rational.

Now we consider the $G$-torus \ref{it:torus}.  Each set $\sG_{n}$
consists of the single graph $G\times C_{n}$.  We label the
cross-sections of the $G$-torus by elements of $\sets{n}$; this
amounts to fixing an initial cross-section and a direction on the
underlying cycle $C_{n}$.  The construction of $D$ is very similar to
the $G$-cylinder, but we must use larger sets $\sS$, $\tilde{\sS}$.
We take $\sS$ (respectively, $\tilde{\sS}$) to be the set of all
$G$-cylinders of length $4$ (resp., $5$) equipped with all possible
colorings.  Thus $|\sS | = 16^{|G|}$ and $\tilde{\sS} = 32^{|G|}$.  If
we write elements $S\in \tilde{\sS}$ as being supported on the graph
$G \times \sets{5}$, then the back and front operators correspond to
selecting certain subgraphs, along with their induced colorings:
\begin{equation}\label{eq:bf}
\tilde{S}_{b} = G \times \{1,2,4,5 \}, \quad \tilde{S}_{f} = G \times \{2,3,4,5 \}.
\end{equation}
In other words, we pass from a $D$-vertex $S$ to a new one $S'$ through
$\tilde{S}$ by (i) inserting a new vertex between the two middle
vertices of $S$ to make $\tilde{S}$, and (ii) taking the last four
vertices of $\tilde{S}$ to make $S'$.

The digraph $D$ and matrix $A$ are constructed the same way as the cylinder case.
The initial vector $v_{0}$ has entries indexed by $\sS$; the entry
labeled $S$ is the total contribution to $G\times C_{4}$ of the
coloring corresponding to $S$.  We then claim
\begin{equation}\label{eq:gf3}
\Xi_{\sG} (x) = W_{G\times C_{3}} (t)x^{3} + x^{4}
\sum_{i=1}^{p}M_{i} (D, v_{0}).
\end{equation}
Note that the generating function \eqref{eq:gf3} starts at $G\times
C_{3}$ since we do not allow loops or multiple edges.  The
verification that \eqref{eq:gf3} is indeed $\Xi_{\sG} (x)$ is very
similar to that for the $G$-cylinder, so we omit the details.
Similarly the proof for the length $n$ earring \ref{it:earring} 
 is very similar to that of the extrusion of $G$ along $H$.

Finally we must prove that the subdivision graphs \ref{it:oneedge},
\ref{it:lastone} form rational families.  Let $e$ be a fixed edge of
$G$, and recall that $S_{n} (G,e)$ is the graph obtained by subdividing
$e$ by adding $n$ new vertices.  To lighten notation we write $S_{n}$
for $S_{n} (G,e)$.  Then the desired family is $\sG_{n} = \{S_{n} \}$.
Instead of the series $\Xi_{\sG} (x)$ we consider 
\[
\Xi '_{\sG } (x) = \sum_{n\geq 4} W_{S_{n}} (t)x^{n}.  
\]

Clearly to prove that $\Xi_{\sG}$ is rational it suffices to prove
that $\Xi '_{\sG}$ is rational.  We do this by using a variation on
the $G$-torus setup.  We take the set $\sS$ (respectively,
$\tilde{\sS} $) to be the path graph $P_{4}$ (resp.,~path graph $P_{5}$)
equipped with all possible colorings, and put $p = |\sS | =
16$. As with the $G$-torus \ref{it:torus}, we identify the vertices of any
$\tilde{S}\in \tilde{\sS}$ with $\sets{5}$, and use front and back
operators defined by
\[
\tilde{S}_{b} = \{1,2,4,5 \}, \quad \tilde{S}_{f} = \{2,3,4,5 \}.
\]
The digraph $D$ and the matrix $A$ are then constructed exactly the
same as in the $G$-torus. 

The vector $v_{0}$ has entries indexed by $\sS$ and computed as
follows.  Let $H \simeq P_{4}$ be the subgraph of $S_{4}$
corresponding to the $4$ new vertices and $3$ new edges added when
passing from $G$ to $S_{4}$.  Then
the entry of $v_{0}$ indexed by $S\in \sS$ is the restricted
$W$-polynomial
\[
W_{S_{4}, H, c (H)} (t),
\]
where $c (H)$ is the fixed coloring of $H$ corresponding to $S$.  We
then have 
\[
\Xi '_{\sG} (x) = x^{4}\sum_{i=1}^{p} M_{i} (D, v_{0}),
\]
which completes the proof of \ref{it:oneedge}. 

Finally we prove \ref{it:lastone}, which corresponds to subdividing
an arbitrary fixed subset $F = \{e_{1},\dotsc ,e_{m} \}$ of the edges
of $G$.  Let us call $m = |F|$ the \emph{rank} of the family
$\sG$. Instead of \eqref{eq:subdividefamily} we consider the family
$\sG ' = \{\sG '_{n} \}$ where 
\begin{equation}\label{eq:restsubdivide}
\sG'_{n} = \bigcup_{\substack{{n_{1},\dotsc ,n_{m}}\\ n_{i}\geq 4}} \bigcup_{k=1}^{m} S_{n_{k}}( G,e_{k}),
\end{equation}
where the outer union is taken over all $n_{1}+\dotsb +n_{m} = n$. It
suffices to prove that $\sG '$ is a rational family, since the general
case, up to finitely many graphs, is a union of lower rank families of
type $\sG '$, and \ref{it:oneedge} establishes the result for families
of rank $1$. 

We now describe the digraph $D$ for the rank $m$ family \eqref{eq:restsubdivide}.
Let $\sS$ and $\tilde{\sS} $ be as in the single-edge case
\ref{it:oneedge}, and let $D_{1}$ be the corresponding digraph.  Thus $|\sS | = 16$.  Then the vertices for the
digraph $D$
correspond to element of the set 
\[
\sS^{m} \times \sets{m} = \{(S_{1},\dotsc ,S_{m}, i) \mid S_{k}\in \sS
, 1\leq i \leq m \}.
\]
Thus $D$ has $p$ vertices where $p = 16^{m}\cdot m$.  We say that the
vertices of the form $(*, \dotsc , *, i)$ are at \emph{level $i$}.

The edges in $D$ come in two types:
\begin{enumerate}
\item [(i)] We have an edge from the level $i$ vertex 
\[
(S_{1}, \dotsc , S_{i}, \dotsc , S_{m}, i)
\]
to another vertex of level $i$ 
\[
(S_{1}, \dotsc , S_{i}', \dotsc , S_{m}, i)
\]
if there is an edge in $D_{1}$ from $S_{i}$ to
$S_{i}'$.  The weight of the edge in $D$ is the same as that of the
edge in $D_{1}$.
\item [(ii)] We have an edge from the level $i$ vertex 
\[
(S_{1}, \dotsc , S_{m}, i)
\]
to the level $i+1$ vertex 
\[
(S_{1}, \dotsc , S_{m}, i+1)
\]
for any $i=1,\dotsc ,m-1$.
The weight of this edge equals $1$.
\end{enumerate}
The edges of type (i) correspond to subdividing the edge $e_{i}$ by
adding a new vertex, whereas the edges of type (ii) correspond to being
finished subdividing $e_{i}$ and moving to a state in preparation for
subdividing $e_{i+1}$.

Now we describe the vector $v_{0}$.  It has entries indexed by 
\begin{equation}\label{eq:st1}
(S_{1}, \dotsc , S_{m}, i).
\end{equation}
We put all entries of $v_{0}$ to be $0$ except for those with $i=1$ in
\eqref{eq:st1}.  In this case, note that each $S_{k}$ in the tuple
corresponds to a coloring $c_{k}$ of the initial $4$ points added to
the edge $e_{k}\in F$.  We set the entry 
\[
(S_{1}, \dotsc , S_{m}, 1) 
\]
in $v_{0}$ to be the restricted $W$-polynomial
\[
W_{S_{4} (F), H, c} (t),
\]
where
\begin{itemize}
\item $S_{4} (F)$ is the graph resulting from adding $4$ new vertices
to each edge in $F$;
\item $H$ is the subgraph of $S_{4} (F)$ that corresponds to the sets
of $4$ inner vertices and $3$ inner edges of the subdivided edges in
$F$; in particular
\begin{equation}\label{eq:st}
H \simeq \coprod_{k=1}^{m} P_{4};
\end{equation}
and 
\item $c$ is the coloring of $H$ that colors the $k$th factor in
\eqref{eq:st} according to $S_{k}$ in \eqref{eq:st1}.
\end{itemize}
This completes the description of $D$ and $v_{0}$.

Finally we must write the generating function $\Xi'$ for the family
\eqref{eq:restsubdivide}.  We claim 
\begin{equation}\label{eq:finalsum}
\Xi_{\sG}' (x) = \sum_{i=p (m-1)/m + 1}^{p} \sum_{N\geq m-1}
(A^{N}v_{0})_{i} x^{N+3m+1}.
\end{equation}
Note that the sum over $i$ in \eqref{eq:finalsum} goes over the
entries of $A^{N} v_0$ corresponding to the vertices labeled
$(S_{1},\dotsc ,S_{k}, m)$ of $D$, i.e.~the vertices at level $m$.
The sum over $N$ begins at $m-1$, since this is the minimal number of
times we must apply $A$ to move from the states at level $1$ to the
states at level $m$.  The power of $x$ is chosen so that $x^{n}$ means
$n$ new vertices have been added.  Up to simple modifications the
expression \eqref{eq:finalsum} is of the same type as those in Theorem
\ref{thm:tmm}, so is a rational function in $x$.  This completes the
proof of the theorem.
\end{proof}

\begin{example}\label{ex:pathgraph}
We illustrate the theorem by giving the matrix $A$ and the vector
$v_{0}$ for the path graph family $G_{n} = P_{n}$.  This is the case of
the $G$-cylinder \ref{it:cylinder} with $G$ an isolated vertex.  The
set $\sS$ consists of the four different colorings of $P_{2}$, namely
$bb$, $bw$, $wb$, and $ww$.  These index the columns and the rows of
the matrix $A$ (see Figure \ref{fig:pathgraph}).  The set
$\tilde{\sS}$ consists of the eight different colorings of $P_{3}$.

Consider for example the first column of $A$, which is indexed by
$bb$.  The entry $1$ in row $bw$ indicates that there is a coloring of
$P_{3}$, namely the first two vertices black and the third white, such
that the corresponding $G$-cylinder $\tilde{\sS}$ has (i)
$\tilde{\sS}_{b}$ colored $bb$, (ii) $\tilde{\sS}_{b}$ colored $bw$,
and (iii) when passing from a coloring of $P_{n}$ ending in $bb$ to a
coloring of $P_{n+1}$ ending in $bw$, the contribution to the
$W$-polynomial does not change.  There are two nonzero entries in this
column because there are two ways to extend the $bb$ coloring of
$P_{2}$ to a coloring of $P_{3}$ with the first two cross-sections
colored $bb$.  

By contrast, the $t$ in column $bw$ and row $wb$ indicates that when
passing from a coloring of $P_{n}$ ending in $bw$ to a coloring of
$P_{n+1}$ ending in $wb$, the contribution to the $W$-polynomial does
change, since an even white vertex has been created.  Again there are
two nonzero entries in this column, corresponding to the two
extensions of $bw$ to $bwb$ and $bww$.  In both cases an even white
vertex is created, although these vertices are not the same in the two
cases.

Note that the sum of the entries of $v_{0}$ is $t^{2}+3$, which is the
$W$-polynomial of the graph $P_{2}$.  The vector $Av_{0}$ is $(2, 2,
2t, t+t^{2})$; its entries sum to the $W$-polynomial of $P_{3}$.

\begin{figure}
\centering \resizebox{0.7\textwidth}{!}{ \input{pathgraphtikz} }
\caption{\label{fig:pathgraph} The matrix $A$ and the vector $v_0$ for
the path graph family.}
\end{figure}

\end{example}

\begin{example}\label{ex:grids}
If for the $G$-cylinder we take $G$ to be a path graph $P_{m}$, then
the family $P_{m}\times P_{n}$, $n\geq 1$ gives the collection of
\emph{grid graphs} $\{G_{m,n} \}$ of fixed depth $m$.  Grid graphs
appear naturally in the context of quantum information. In particular,
the associated grid graph quantum states arise naturally as resources
for measurument-based quantum computation, see
e.g.~\cite{jozsa2006introduction, briegel2009measurement}.  By Theorem
\ref{thm:ratfamilies} this is a rational family.  For example, the
family $\{G_{2,n} \}$ of grid graphs of depth $2$ has the
$W$-polynomials shown in Table \ref{tab:g2}.  This family is rational
and its generating function is
\begin{equation}\label{eq:g2rf}
\frac{
1 - 4(t-1)^2x^2 - 3(t^2-1)^2x^3 + 4(t-1)^4(t+1)^2x^5}{1 - (t^2+3)x
+ 2(t^{2}-1)^2x^3 + (t^{2}-1)^2(t^{2}+3)x^4 - (t^{2}-1)^4x^6}.
\end{equation}

We remark that in the construction \ref{it:cylinder}, the digraph $D$ has order $16$, so $A$ is $16\times 16$ and the
denominator of \eqref{eq:g2rf} could have been degree $16$ in $x$.
Experimentally one finds that the denominator is considerably smaller: for the grid family of depth $m$, the denominator from the transfer matrix has degree $3^m$ in $x$, and the numerator has degree $3^{m-1}$. However this is still not the ``true'' denominator, since there is considerable cancellation.  Thus the denominator of
\eqref{eq:g2rf} corresponds to a different recursive computation
of $W_{G_{2,n}} (t)$ than that of Theorem \ref{thm:ratfamilies}.  We
have not checked how to describe this recursion combinatorially; nor
do we know the degree of the actual denominator of the rational
functions in Theorem \ref{thm:ratfamilies}.  

One can ask if the $W$-polynomials of the full collection of grid
graphs $\{G_{m,n} \mid m,n\geq 1\}$ could be grouped into a bivariate
generating function that is rational.  We strongly doubt that this is
possible, since although each family for fixed $m$ is rational, the
degrees of the denominators increase rapidly with $m$.

\begin{table}[htb]
\begin{center}
\begin{tabular}{|c|l|}
\hline
$n$& $W_{G_{2,n}} (t)$\\
\hline
0 &   $1$\\
1 &   $t^2 + 3$\\
2 &   $t^4 + 2t^2 + 8t + 5$\\
3 &   $t^6 + t^4 + 9t^3 + 20t^2 + 23t + 10$\\
4 &   $t^8 + t^6 + 12t^5 + 23t^4 + 48t^3 + 83t^2 + 68t + 20$\\
5 &   $t^{10} + t^8 + 15t^7 + 27t^6 + 67t^5 + 160t^4 + 253t^3 + 278t^2 + 177t + 45$\\
\hline
\end{tabular}
\end{center}
\caption{\label{tab:g2}}
\end{table}
\end{example}

\begin{example}\label{ex:subdivide}
We give an example of the subdivision family \ref{it:lastone}.  Let
$G$ be the path graph $P_{3}$ and let $F = E (G) = \{e_{1}, e_{2} \}$.
We consider the family $\sG '$ where each edge is subdivided a minimum
of four times.  Note that the first set of graphs in the family is
$\sG_{4} = \{P_{11} \}$, since after subdividing both edges four times
we obtain $P_{11}$. Similarly $\sG '_{5}$ is two copies of $P_{12}$,
since we have two ways to subdivide one edge in $F$ five times and the
other four times.  In general, one sees that for $n\geq 4$ the set $\sG
'_{n}$ consists of $n-3$ copies of $P_{n+7}$.

Now we consider the digraph $D$.  It has $p = 16^{2}\cdot 2 = 512$
vertices, the adjacency matrix $A$ is $p\times p$ and the vector
$v_{0}$ is $p$-dimensional.  Only the first $p/2 = 256$ entries of
$v_{0}$ are nonzero.  For any $p$-dimensional vector $v$ let $\rho (v)
= \sum_{i=p/2+1}^{p} v_{i}$ (this is the sum of the last $p/2$ of the
entries of $v$).

\begin{itemize}
\item Consider the sum $\rho (A v_{0})$.  We claim this is the
polynomial $W_{P_{11}} (t)$.  Indeed applying $A$ once corresponds to
taking one step in the digraph $D$.  To get a contribution to $\rho
(Av_{0})$ one must step from the vertices at level $1$ to those at
level $2$.  This corresponds to adding no new vertices to $P_{11}$,
and thus the result is $W_{P_{11}} (t)$.
\item We claim $\rho (A^{2}v_{0})$ is $2W_{12} (t)$.  Indeed
contributing to this
corresponds to taking two steps in the digraph $D$ and ending at a
level $2$ vertex.  One can step first within level $1$ (subdivide
$e_{1}$) and then down to level $2$ (don't subdivide $e_{2}$) or can step
down to level $2$ (don't subdivide $e_{1}$) and then step within level
$2$ (subdivide $e_{2}$).  After this one has subdivided either $e_{1}$
or $e_{2}$ to $5$ vertices, so the result is $2$ copies of $P_{12}$.

\item One can similarly check that the polynomial $\rho (A^{n}v_{0})$
is $n W_{P_{11+n-1}} (t)$. 
\end{itemize}
The final series is 
\begin{multline*}
\Xi_{\sG}' (x) = \\
(t^{11} + 2t^9 + 13t^8 + 24t^7 + 58t^6 + 146t^5 \\
+ 308t^4 + 519t^3 + 566t^2 + 332t + 79)y^8 \\
+ (2t^{12}
+ 4t^{10} + 28t^9 + 52t^8 + 128t^7 + 336t^6 \\
+ 728t^5 + 1426t^4 + 2144t^3 + 2044t^2 + 1068t + 232)y^9 + \dotsb \\
\end{multline*}
and the rational function has numerator
\begin{multline*}
(t^{13} + 11t^{10} + 17t^9 + 25t^8 + 72t^7 + 78t^6 - 77t^5 -
238t^4 - 136t^3 + 87t^2 + 123t + 37)y^{12} \\
+ (2t^{12} + 20t^9 +
30t^8 + 48t^7 + 96t^6 - 8t^5 - 242t^4 - 208t^3 + 64t^2 + 148t
+ 50)y^{11} \\
+ (-2t^{12} + t^{11} - 2t^{10} - 24t^9 - 31t^8 - 66t^7 -
192t^6 - 372t^5 - 404t^4 + 9t^3 + 506t^2 + 452t + 125)y^{10} 
\\
+ (-2t^{11} - 2t^9 - 22t^8 - 36t^7 - 72t^6 - 180t^5 - 228t^4 -
26t^3 + 248t^2 + 246t + 74)y^9 \\
+ (t^{11} + 2t^9 + 13t^8 + 24t^7 +
58t^6 + 146t^5 + 308t^4 + 519t^3 + 566t^2 + 332t +
79)y^8
\end{multline*}
and denominator
\[
(t^4 - 2t^2 + 1)y^6 + (-2t^3 - 2t^2 + 2t + 2)y^4 +
(2t^2 - 2)y^3 + (t^2 + 2t + 1)y^2 + (-2t - 2)y + 1.
\]
\end{example}

\section{The black-white polynomials of all graphs}\label{s:feynman}

\subsection{} In this section we describe a method to compute the
exponential generating function of the $W$-polynomials over \emph{all}
connected graphs.  Here we allow graphs to have loops and parallel
edge sets.  The extension of the definition of $W (t)$ to such graphs
is done in the obvious way.  First of all, loops do not change the
parity of a vertex, so do not affect the definition of admissibility
of a vertex.  Otherwise, suppose the color of a vertex $v$ is 0 and
that $w_{1},\dotsc ,w_{k}$ is the full set of vertices joined to $v$
with $c (w_{i})=1$, $i=1,\dotsc ,k$.  Suppose there are $m_{i}$
parallel edges going between $v$ and $w_{i}$, $i=1,\dotsc ,k$.  Then
$v$ is admissible if either $c (v) = 1$, or if $c (v) = 0$ and
\[
\sum_{i=1}^{k}m_{i} = 0 \bmod 2.
\]
With this modification, the definition of $W$ proceeds exactly as in
Definition \ref{def:defofW}.

\subsection{}
To define our generating function, we need more notation.  Let $u,
\xi_{1}, \xi_{2}, \dotsc$ be indeterminates.  Our generating function
will ultimately live in the ring 
\[
\Q [t, \xi _{1}, \xi_{2}, \dotsc  ]\fps{u},
\]
i.e.~it will be a formal power series in $u$ with coefficients
rational polynomials in $t$ and the $\xi_{i}$ (by definition, for any
element in the infinite polynomial ring only finitely many of the
$\xi_{i}$ can appear).

Let $\n = (n_{1},n_{2},\dotsc)$ be a vector of nonnegative integers,
with $n_{k}$ nonzero only for finitely many $k$.  Let $|\n | = \sum k
n_{k}$. We say a graph $G$ has \emph{profile $\n$} if it has $n_{k}$
vertices of degree $k$.  Let $G (\n)$ be the set of all graphs of
profile $\n$, up to isomorphism, and let $G_{c} (\n)$ be the connected
graphs.  We note that $G (\n)$ is a
\emph{finite} set.  For any $G \in G (\n)$, let $\Aut G$ be its
automorphism group.  This includes all the automorphisms induced from
permuting vertices, but it also includes automorphisms that permute
sets of parallel edges and flip the half-edges of loops; such
automorphisms act trivially on the vertices. 

\begin{definition}\label{def:bigseries}
The \emph{generating function for the $W$-polynomials of all graphs}
is
\begin{equation}\label{eq:Wgraphseries}
\sB (u) = \sum_{\n } u^{|\n |} \sum_{G \in G_{c} (\n
)}\frac{W_{G}(t)\prod_{k} \xi_{k}^{n_{k}}}{|\Aut G|}.
\end{equation}
\end{definition}

\subsection{}
The main result of this section is Theorem \ref{thm:genfunallW}, which gives an expression for $\sB(u)$ using the technique of Feynman
diagrams.  This technique allows one to build graph generating functions with
various restrictions on the types of graphs, like valences of vertices
taken from a fixed set of positive integers, bipartite graphs, and so
on.  For more details about the technique we refer to Etingof
\cite{etingof}.  As a warm-up, in section \ref{ss:allgraphs} we first describe how to build the
series $\sB (u)$ where all the $W$-polynomials are set to $1$,
i.e.~the generating function of all graphs weighted by their
automorphisms (Theorem~\ref{thm:genfunallG}).  The key point is the combinatorial interpretation of
the Gaussian moments.

\subsection{}\label{ss:allgraphs}
Let $x$ be a real variable.  Suppose we normalize the measure $dx$ on
$\RR$ such that $\int_{-\infty}^{\infty} e^{-x^{2}/2} \, dx =
1$. Consider the moment
\begin{equation}\label{eq:gaussianmoment}
\ip{x^{k}} :=
\int_{-\infty}^{\infty} x^{k} e^{-x^{2}/2}\, dx.
\end{equation}
If $k$ is odd, we have $\ip{x^{k}}=0$.  If $k$ is even, then it is
well known that $\ip{x^{k}}$ is the \emph{Wick number}
\[
W_{2} (k) = \frac{k!}{2^{k/2} (k/2)!},
\]
which counts the pairings of the finite set $\sets{k} := \{1,\dotsc ,k
\}$.  Let $S (x)$ be the formal power series $\sum_{k\geq 1}
\xi_{k}x^{k}/k!$.  Using \eqref{eq:gaussianmoment}, we can evaluate the
integral
$\ip{\exp (S (ux))}$ as a formal
power series in $u$ with coefficients in the polynomial ring $\QQ
[\xi_{1}, \xi_{2}, \dotsc ]$: we simply replace each power of $x$ with
the relevant Wick number.  After doing this we find 
\begin{equation}\label{eq:ipSg}
\ip{\exp S (ux)} = 1 +A_{2}u^{2}/2 + A_{4}u^{4}/8
+A_{6}u^{6}/48 + \dotsb 
\end{equation}
where the first few coefficients (\footnote{The polynomials $A_{k}$
are related to the Bell polynomials \cite{comtet}.}) are given by
\begin{multline*}
A_{2} = \xi_{1}^{2}+\xi_{2}, \quad A_{4} = \xi_1^4 +
6\xi_{1}^{2}\xi_2 + 4\xi_{1}\xi_{3} + 3\xi_2^2 + \xi_4,\\
A_{6} = \xi_1^6 + 15\xi_{1}^{4}\xi_2 + 20\xi_1^3\xi_{3} + 
45\xi_{1}^{2}\xi_2^2 
+ 15\xi_{1}^{2}\xi_4 \\
+ 60\xi_{1}\xi_{2}\xi_{3} + 
6\xi_{1}\xi_5 + 15\xi_2^3 + 15\xi_{2}\xi_4 + 10\xi_3^2 + \xi_6.
\end{multline*}
We claim that \eqref{eq:ipSg} can be interpreted as a generating function for graphs,
inversely weighted by the orders of their automorphism groups.
Namely, we have the following theorem:

\begin{theorem}\label{thm:genfunallG}
We have
\begin{equation}\label{eq:genfunallG}
\ip{\exp S (ux)} = \sum_{\n } u^{|\n |} \sum_{G \in G (\n
)}\frac{\prod_{k} \xi_{k}^{n_{k}}}{|\Aut G|}.
\end{equation}
The corresponding sum over connected graphs is given by $\log \ip{\exp S (ux)}$.
\end{theorem}

\begin{proof}
We give a sketch of the proof, since it helps to understand the
construction of $\sB (u)$ and the proof of Theorem \ref{thm:genfunallW}.  Full details can be found in
\cite{etingof}.  We define a \emph{$k$-flower} $F_{k}$ to be a vertex incident
to $k$ half-edges (see Figure~\ref{fig:flowers}).  A $k$-flower has $k!$ automorphisms, given by
freely permuting the half-edges.  Any graph is formed by taking a
collection of $k$-flowers for various $k$ and choosing a complete
pairing of the half-edges, and the automorphisms of such a graph come
from permuting the flowers of the same degrees and the joined edges.

Thus to build the series on the right of \eqref{eq:genfunallG}, we
must take all possible finite sets of flowers, build all possible edge
pairings, and then weight by the correct order of the automorphism group.
We claim this is exactly what is accomplished by the left of
\eqref{eq:genfunallG}.  Indeed, mark a $k$-flower with the variable
$\xi_{k}$.  Then the exponential series
\[
S (ux) = \sum_{k\geq 1}
\xi_{k}u^{k}\frac{x^{k}}{k!}
\]
represents independently choosing a $k$-flower for each $k$, with $x$
and $u$ both marking the number of available half-edges to pair (see
Figure \ref{fig:flowergluing}).  The
expression 
\[
\exp S(ux) = \exp\Bigl(\sum_{k\geq 1}
\xi_{k}u^{k}\frac{x^{k}}{k!}\Bigr)
\]
then gives the exponential series of all possible choices of finite
sets of flowers.  The $\xi_{i}$ keep track of the different types of
flowers in a finite set, while $x$ and $u$ both keep track of the total
half-edges in a given set.  When we evaluate the expectation 
\[
\ip{\exp S (ux)}
\]
we eliminate $x^{n}$ from each term and replace it with the number of
pairings on $\sets{n}$.  The monomials in the $\xi_{i}$ survive to record the
profile, the power of $u$ records the resulting number of half-edges
(i.e.~twice the number of edges), and the
orbit-stabilizer theorem means that we are left with the correct order
of automorphisms.  This proves \eqref{eq:genfunallG}.  To get the sum
over connected graphs we take the logarithm, which proves the second
statement.

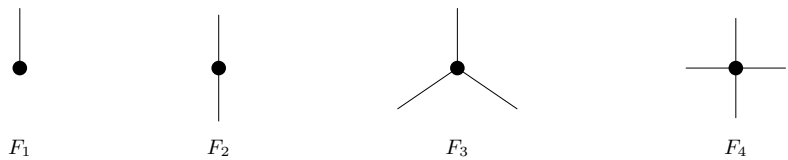
\begin{figure}[htb]
\centering
\resizebox{0.85\textwidth}{!}{
\begin{tikzpicture}[font=\small, line cap=round, line join=round]

\def\r{0.11}      
\def\yDot{-1.0}   
\def\arm{0.9}     
\def\labY{-2.2}   

\def\xone{0}
\def\xtwo{3.0}
\def\xthree{6.6}
\def\xfour{10.8}

\begin{scope}[shift={(\xone,0)}]
  \draw (0,\yDot+\arm) -- (0,\yDot);
  \fill (0,\yDot) circle (\r);
\end{scope}

\begin{scope}[shift={(\xtwo,0)}]
  \draw (0,\yDot+0.8) -- (0,\yDot-0.8);
  \fill (0,\yDot) circle (\r);
\end{scope}

\begin{scope}[shift={(\xthree,0)}]
  \fill (0,\yDot) circle (\r);
  \draw (0,\yDot) -- (0,\yDot+\arm);
  \draw (0,\yDot) -- (-0.9,-0.85-0.765);
  \draw (0,\yDot) -- ( 0.9,-0.85-0.765);
\end{scope}

\begin{scope}[shift={(\xfour,0)}]
  \fill (0,\yDot) circle (\r);
  \draw (0,\yDot+0.75) -- (0,\yDot-0.75);
  \draw (-0.75,\yDot) -- (0.75,\yDot);
\end{scope}

\node at (\xone,\labY)  {$F_1$};
\node at (\xtwo,\labY)  {$F_2$};
\node at (\xthree,\labY) {$F_3$};
\node at (\xfour,\labY) {$F_4$};

\end{tikzpicture}
}
\caption{Some flowers.}
\label{fig:flowers}
\end{figure}

\begin{figure}[htb]
\begin{center}
\includegraphics[scale=0.15]{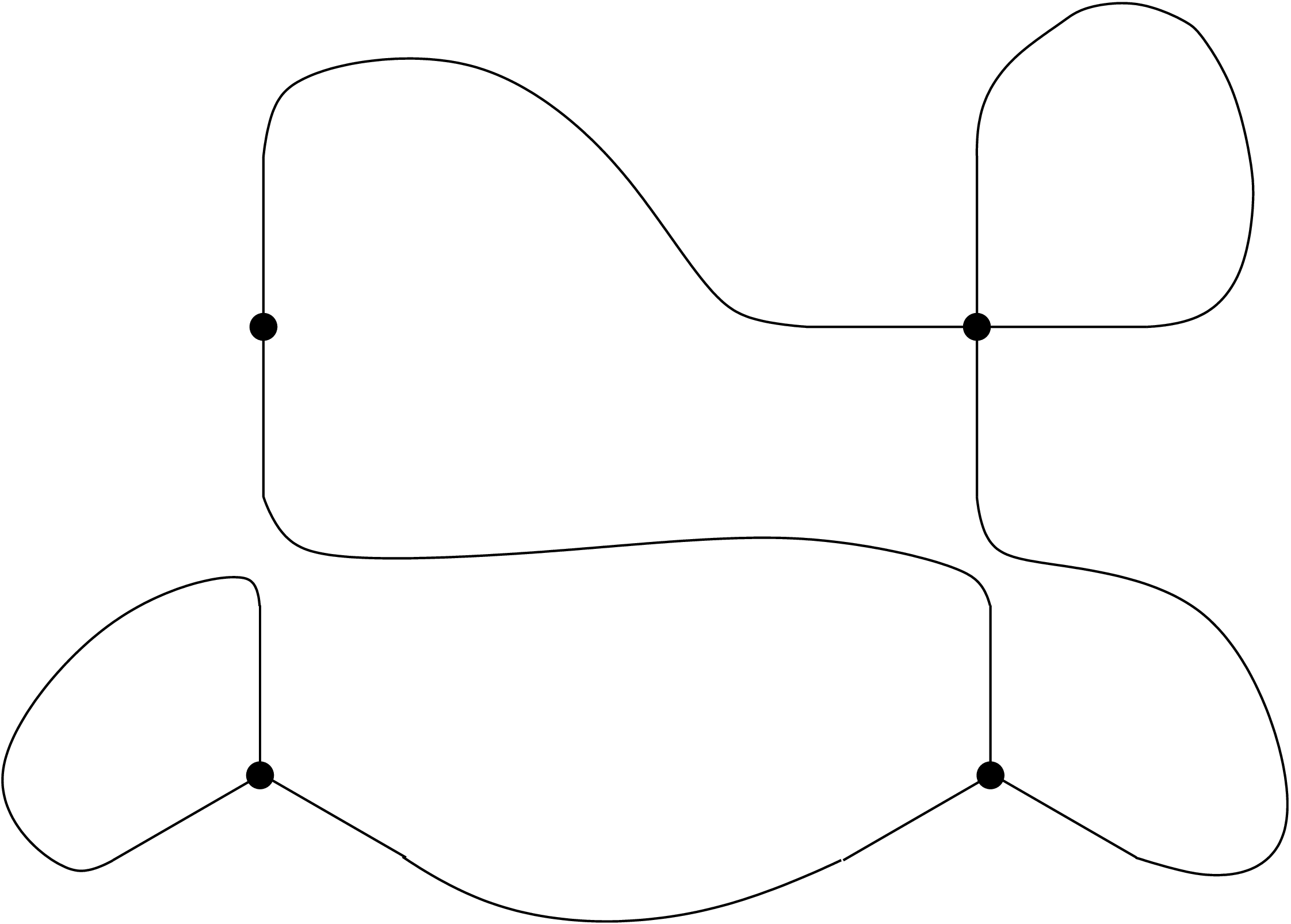}
\end{center}
\caption{Gluing flowers into a graph $G$ with profile $\n =
(0,1,2,1,0,\dotsc)$. We have $|\n| = 1\cdot 2 + 2\cdot 3 + 1\cdot4 =
12$.  The graph has $\Aut G \simeq \Z /2\Z \times \Z /2\Z$, coming
from flipping the two loops.  The corresponding term in the sum $ \log
\ip{\exp S (ux)}$ is $u^{12}\xi_{2}\xi_{3}^{2}\xi_{4}/4$. The full
coefficient of all graphs (including disconnected graphs) with this
profile is $385/128$, which is computed as follows. The expectation
$\ip{x^{12}} = 10395$, so there are this many ways to pair the $12$
half-edges.  Exponentiation produces a denominator of $2!\cdot
2!(3!)^{2}4! = 3456$ for this profile, thus the result is
$10395/3456=385/128$.  \label{fig:flowergluing}}
\end{figure}
\end{proof}

\subsection{}
The proof of Theorem \ref{thm:genfunallG} shows that it is possible to
make many variants of \eqref{eq:genfunallG}.  For instance, if one
wants graphs of a fixed regularity, say $r$, then one can set all
$\xi_{i}=0$ except $\xi_{r}$.  It is not necessary to use the Gaussian
moments, or a probability measure, or even an integral.  One can
define formal ``moments'' by taking an arbitrary linear functional
$\ip{x^{n}} := a_{n}$, $n=1,2,\dotsc $, where $a_{n}$ is in a
commutative ring $A$.  For example, if one fixes $m\geq 2$ and defines
$a_{n}$ to be the number of partitions of $\sets{n}$ into subsets of
order $m$, then the corresponding series $\eqref{eq:genfunallG}$
enumerates $m$-uniform hypergraphs (cf.~\cite{ncontribution}), again
grouped by profile and weighted by inverse automorphism order.

\subsection{}
We now show how to modify the construction of \eqref{eq:genfunallG} to
produce the series $\sB (u)$.  We enlarge the set of flowers
to two types: we take flowers $E_{1}$, $E_{2}$, \dots as before, which
we take to have black vertices.  We also take flowers $F_{1}$,
$F_{2}$, \dots, which we take to have white vertices.  These $k$-flowers
are marked by the variables $\xi_{k}$ as before, independent of
whether they are black or white, and will appear in
the construction weighted by $1/k!$, as before.   

We will also replace the Gaussian measure with a formal ``measure'' that is defined through its ``moments.''
We replace the pairing variable $x$ with some new pairing
variables $b, w, y, z$.  We mark the half-edges of the two
kinds of flowers with these variables.  The half-edges of $E_{k}$ are
marked by any choice of $b$ and $y$, and similarly the half-edges of
$F_{k}$ are marked by any choice of $w$ and $z$.  We modify the
expectation by introducing a formal measure $\ipm{bw}{\phantom{a}}$ on
monomials in the pairing variables corresponding to the following
rules:
\begin{enumerate}
\item Any $b$ (resp.~$w$) variable can be paired only with another $b$
(resp.~$w$) variable.  
\item Any $y$ can only pair with a $z$ and vice versa.
\end{enumerate}
In all cases the measure should count all the possible pairings that
can occur. See Figure ~\ref{fig:formal_measure_explanation} for an illustration. 

\begin{figure}[htb]
\centering
\resizebox{0.7\textwidth}{!}{
\begin{tikzpicture}[scale=1.2, line cap=round, line join=round]
  \coordinate (L) at (-2,0);
  \coordinate (R) at ( 2,0);

    \node at (-5.2, 0.6) {$b$};
    \node at (-5.2, -0.6) {$b$};

    \node at (5.2, -0.6) {$w$};
    \node at (5.2, 0.6) {$w$};

    \node at (3.6, 0.6) {$w$};
    \node at (3.6, -0.6) {$w$};

    \node at (-0.4, 1.4) {$y$};
    \node at (0.4, 1.4) {$z$};

    \node at (-0.4, -1.4) {$y$};
    \node at (0.4, -1.4) {$z$};

  \fill (L) circle (4.2pt);
  \foreach \n in {1}
    \draw[line width=0.8pt] (-2-1.5/\n,0) circle (1.5/\n);

  \draw[line width=0.8pt]
    (L) .. controls (-1.3,0.7) and (-0.6,1.1) .. (-0.2,1.1);
  \draw[line width=0.8pt]
    (L) .. controls (-1.3,-0.7) and (-0.6,-1.1) .. (-0.2,-1.1);

  \foreach \n in {1,2}
    \draw[line width=0.8pt] (2+1.5/\n,0) circle (1.5/\n);

  \draw[line width=0.8pt]
    (R) .. controls (1.3,0.7) and (0.6,1.1) .. (0.2,1.1);
  \draw[line width=0.8pt]
    (R) .. controls (1.3,-0.7) and (0.6,-1.1) .. (0.2,-1.1);

\draw[fill=white] (R) circle (4.2pt);

\draw[fill=white,draw=none] (-3.5, 0) circle (6.2pt);
\draw[fill=white,draw=none] (-5, 0) circle (9.2pt);

\draw[fill=white,draw=none] (3.5, 0) circle (6.2pt);
\draw[fill=white,draw=none] (5, 0) circle (9.2pt);

\end{tikzpicture}
}
\caption{An example depicting a possible pairing of a black  flower $E_4$ and a white flower $F_6$. The formal measure $\ipbw{b^{2}w^{4}y^{2}z^{2}}$ evaluates to $\ipbw{b^{2}}\ipbw{w^{4}}\ipbw{y^{2}z^{2}} = 1\cdot 3\cdot 2 = 6$.}
\label{fig:formal_measure_explanation}
\end{figure}
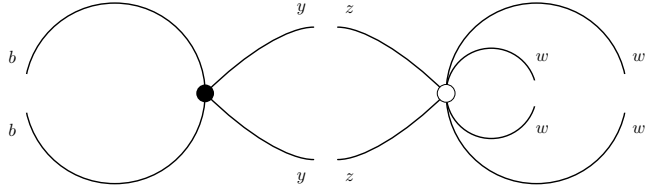

This leads to the following:

\begin{definition}\label{def:measure}
We define the $\Z$-valued linear operator $\ipbw{\phantom{a}}$ on the polynomial
ring $\QQ [b,w,y,z]$ as follows:
\begin{enumerate}
\item on monomials $b^{n}w^{m}y^{k}z^{l}$, we
define  $\ipbw{b^{n}w^{m}y^{k}z^{l}} = \ipbw{b^{n}}\ipbw{w^{m}}\ipbw{y^{k}z^{l}}$;
\item we put $\ipbw{b^{n}} = \ipbw{w^{n}} = W_{2} (n)$ if $n$ is even;
$0$ otherwise; and 
\item we put $\ipbw{y^{k}z^{l}} = 0$ unless $k=l$, in which case
$\ipbw{y^{k}z^{k}}=k!$; and 
\item we extend linearly to all polynomials in $b,w,y,z$.
\end{enumerate}

\end{definition}

We are now ready for the main result of this section: 

\begin{theorem}\label{thm:genfunallW}
For $n\geq 1$, let $p (n) = 1$ if $n$ is even and $t$ if $n$ is odd.
Define $\bF (n) \in \Z [b, w, y, z, t]$ by  
\begin{equation}\label{eq:bigflower}
\bF (n) = \Bigl((b+y)^{n} + \sum_{i=0}^{n}\binom{n}{i}p (i) w^{n-i}
z^{i}\Bigr).
\end{equation}
Put 
\begin{equation}\label{eq:partition}
S (u) = \sum_{k\geq 1} \bF (k)\frac{\xi_{k}u^{k}}{k!} \in R [b, w, y, z, t] \fps{u}
\end{equation}
Then 
\begin{equation}\label{eq:computemoment}
\ipbw{\exp S (u)} = \sum_{\n } u^{|\n |} \sum_{G \in G (\n
)}\frac{W_{G}(t)\prod_{k} \xi_{k}^{n_{k}}}{|\Aut G|},
\end{equation}
where the expectation $\ipbw{\phantom{a}}$ is given in Definition
\ref{def:measure}, 
and we have 
\[
\sB (u) = \log \ipbw{\exp S (u)}.
\]
\end{theorem}

\begin{proof}
The proof is very similar to that of Theorem \ref{thm:genfunallG}.
First of all, the expression \eqref{eq:bigflower} corresponds to all
possibilities of the
following construction:
\begin{enumerate}
\item Choose either the black flower $E_{n}$ or the white flower $F_{n}$.
\item If $E_{n}$, then make a choice of marking each of its $n$ half-edges with
either a $b$ or a $y$.
\item If $F_{n}$, then make a choice of marking each of its $n$ half-edges with
either a $w$ or a $z$.  If the number of $z$s appearing is even, then
mark this $F_{n}$ with a $t$.
\end{enumerate}

Next, the expression \eqref{eq:partition} corresponds to performing
the construction in \eqref{eq:bigflower} for each possible number of
half-edges.  The flowers are then marked by $u$ (respectively, $\xi_{k}$) to
record the total number of half-edges (resp., the degree).
Each term is divided by $k!$ to account for the automorphisms of a flower.

Finally, the construction on the left of \eqref{eq:computemoment} does the
following.  The exponential function produces all possible
combinations of the marked flowers from the previous steps.  When one
computes the moment $\ipbw{\phantom{a}}$ using Definition \ref{def:measure}, one counts the possible ways
to join the flowers together.   Flowers may be 
joined to each other indiscriminately.  The pairings involving $b$
(respectively, $w$) join black to black (resp.,~white to white) and so
do not contribute any power of $t$ to the polynomials $W$.  Such
contributions come from joining a black flower to a white flower, and
only occur if a white flower has been joined to an even number of
black flowers (this is the role of the function $p$: it ensures that white flowers are only joined to an even number of black flowers).   The same
discussion as in the proof of Theorem \ref{thm:genfunallG} shows that
\eqref{eq:computemoment} holds.  Finally the series $\sB (u)$, which
only sums over connected graphs, arises by taking the logarithm of
\eqref{eq:computemoment}. This completes the proof.
\end{proof}

We conclude with some examples.  We set almost all of the $\xi_{k}$ to
$0$, i.e.~we put restrictions on the possible degrees of flowers, to simplify the presentation.

\begin{example}\label{ex:oneonly} 
If we set all but $\xi_{1}$ to $0$, then we obtain all graphs where
every vertex has degree exactly 1.  These correspond to unions of
path graphs $P_{2}$; in the quantum literature such graphs are known as \emph{Bell pairs}.   The resulting series is 
\begin{align*}
\ipbw{\exp (\bF (1)) }&= 1 + (t^2 + 3)\frac{\xi_{1}^{2}u^2}{2} + (t^4
+ 6t^2 + 9)\frac{\xi_{1}^{4}u^4}{8} + \dotsb \\
&= \sum_{k\geq 0} (t^{2}+3)^{k}\frac{\xi_{1}^{2k}u^{2k}}{2^{k}k!}.
\end{align*}
The $W$-polynomial $(t^{2}+3)^{k}$ corresponds to $k$ copies of
$P_{2}$.  The $\xi_{1}^{2k}$ records that $2k$ flowers of degree $1$
were used to build these path graphs, and the $u^{2k}$ records that we
glued together $2k$ half-edges in total.  Finally the denominator
$2^{k}k!$ indicates that such a graph has automorphisms corresponding
to independently flipping the $P_{2}$s and permuting them.

The series of connected graphs only has one term, namely
$(t^{2}+3)\xi_{1}^{2}u^{2}/2$.  This corresponds to a single copy of
$P_{2}$. We get $(t^{2}+3)/2$ since
$W_{P_{2}} (t) = t^{2}+3$ and $|\Aut P_{2}| = 2$.
\end{example}

\begin{example}\label{ex:onetwoonly} 
If we set all but $\xi_{1}$ and $\xi_{2}$ to $0$, then we obtain all graphs where
every vertex has degree either 1 or 2.  This means we can build
disjoint unions of path graphs and cycles.  The resulting series of
connected graphs is 
\begin{multline*}
\log \ipbw{\exp (\bF (1)) + \bF (2)} = \\
((t^2+3)\xi_1^2 + (t+1)\xi_2) \frac{u^2}{2} \\
%
+((2t^{3}+6t+2)\xi_{1}^{2}\xi_{2} + (t^{2}+2t+1) \xi_{2}^{2}) \frac{u^{4}}{4}\\
+((3t^{4}+6t^{2}+24t +
15)\xi_{1}^{2}\xi_{2}^{2}+(t^{3}+3t+4)\xi_{2}^{3})\frac{u^6}{6} + \dotsb 
\end{multline*}
For example, in the $u^{2}$ term the coefficient of $\xi_{1}^{2}$
corresponds to the path graph $P_{2}$.    Similarly the
coefficient of $\xi_{2}$
corresponds to a $1$-cycle, i.e.~a vertex with a single loop.
\end{example}

\begin{example}\label{ex:flower3}
Suppose we set all $\xi_{i}=0$ except $\xi_{3}$.  Then all vertices
have degree $3$.  The series begins
\begin{multline*}
\log \ipbw{\exp (\bF (3))} =
\frac{5}{24}\bigl(t^{2}+3\bigr)\xi_3^2u^{6} + \frac{5}{16}\bigl(t^4
+24t^{2} + 3\bigr)\xi_{3}^4u^{12} + \dotsb 
\end{multline*}
The coefficient of $\xi_{3}^{2}u^{6}$ comes from the two connected
graphs with two vertices of degree 3, which are the graphs $G_{1}$ and
$G_{3}$ in Figure \ref{fig:gamma}.  Both have $W$-polynomial
$t^{2}+3$, and we have $1/12+1/8=5/24$.  For the coefficient of
$\xi_{3}^{4}u^{12}$, there are five connected graphs; their
automorphism groups have orders $24$, $48$, $16$, $16$, and $8$ (the
first is the complete graph $K_{4}$).  We leave the pleasure of
finding them, and computing their $W$-polynomials, to the reader.
\end{example}
 
\section{Wright series of black-white polynomials}\label{s:wright}
\subsection{}
In this section we explain how to produce generating functions of
$W$-polynomials for connected graphs of a given loop number $g \geq
0$, that is, connected graphs whose first Betti number (rank of
$H_{1}$) is $g$.  As in Section \ref{s:feynman}, we allow graphs to have
loops and parallel edge sets.  We begin by recalling work of Wright
\cite{wright1}.  For more other expositions of this material, we
refer to Flajolet--Sedgewick \cite[II.5]{fs} and
Dubrovin--Yang--Zagier \cite[\S 3.1]{dyz}.

\subsection{}
For $g\geq 0$ let $\Gamma(g)$ be the set of connected graphs $G$ with
first Betti number $g$.  Let $v (G)$ (respectively, $e (G)$) be the
number of vertices (resp.,~edges) of $G$. Let $\Aut G$ be the full
automorphism group of $G$ generated by permuting vertices and parallel
edges, as well as flipping and permuting loops.  Then Wright considered
the series 
\begin{equation}\label{eq:Aser}
\sA_{g} (x) = \sum_{v\geq 1} \sum_{G\in \Gamma (g)} \frac{x^{v (G)} }{|\Aut G|}.
\end{equation}
For example
\begin{align*}
\sA_{0} &= x + x^2/2 + x^3/3 + 2x^4/3 + 25x^5/24 + 9x^6/5 + 2401x^7/720 + \dotsb, \\
\sA_{1} &= x/2 + 3x^2/4 + 17x^3/12 + 71x^4/24 + 523x^5/80 + 899x^6/60 + \dotsb, \\
\sA_{2} &= x/8 + 7x^2/12 + 101x^3/48 + 83x^4/12 + 12487x^5/576 + 3961x^6/60 + \dotsb .
\end{align*}

Wright showed that for $g\geq 1$ the series \eqref{eq:Aser} can be
computed in terms of what is essentially the series for the $g=0$
case.  More precisely, let $\sL$ be defined through the recursive
relation 
\begin{equation}\label{eq:lfunction}
\sL = x \exp (\sL),
\end{equation}
or alternatively as the the inverse
series of $xe^{-x}$:
\begin{equation}\label{eq:Tintro}
\sL (x)  = x+x^{2}+3x^{3}/2 + 8x^{4}/3 + \dotsb .
\end{equation}
The series $\sL$, which is closely related to the Lambert function, is
the generating function of \emph{rooted trees} weighted by the inverse
orders of their automorphism groups.  It can be computed to any
desired precision by solving iteratively, or can even be computed
exactly using Lagrange inversion. This is almost the series $\sA_0$,
which enumerates unrooted trees; in fact we have
\begin{equation}\label{eq:W0}
\sA_{0} (x) = \int x^{-1}\sL (x)\,dx.
\end{equation}

All the other $\sA_g$ can be given in terms of $\sL$.  Let
$\Gamma_{\geq 3} (g) \subset \Gamma (g)$ be the subset of connected
loop number $g$ graphs with all vertices having degree at least $3$.
It is easy to see that this set is finite for all $g$.  Then if $g\geq
2$, we have
\begin{equation}\label{eq:Wgintro}
\sA_{g} (x) = \sum_{G\in \Gamma_{\geq 3} (g)} \frac{1}{|\Aut G|}\frac{\sL^{v
(G)}}{(1-\sL)^{e (G)}}, \quad g\geq 2.
\end{equation}
Thus $\sA_{g}$, $g\geq 2$ is a \emph{rational function} in the tree
series $\sL$.  This is almost as good as the series being a rational
function in $x$, but of course is not the same.(\footnote{One doesn't
expect $\sA_g$ to be a rational function in $x$.})

\subsection{}
The expression \ref{eq:Wgintro} has a simple geometric interpretation.
A graph $H\in \Gamma (g)$ determines a unique graph $G\in \Gamma_{\geq
3} (g)$: one simply deletes all vertices of degree $1$ until none
remain, and then deletes all vertices of degree $2$ and replaces them
with an edge joining their neighbors (see Figure \ref{fig:reduction}).
We call $G$ the \emph{reduction} of $H$.  Conversely, any graph in $\Gamma
(g)$ can be obtained by reversing this process.  This is exactly
captured in the expression \ref{eq:Wgintro}:
\begin{enumerate}
\item We sum over all possible reductions $G$.
\item For each $G$ we graft an arbitrary rooted tree onto any of its
vertices.  This corresponds to the factor $\sL^{v (G)}$ in the
numerator.  
\item For each $G$ we can also freely subdivide its edges and then
graft an arbitrary rooted tree onto any of the new edges.  This is the
factor $1/(1-\sL)^{e (G)} = (1+\sL +\sL^{2} + \dotsb )^{e (G)}$.
\end{enumerate}

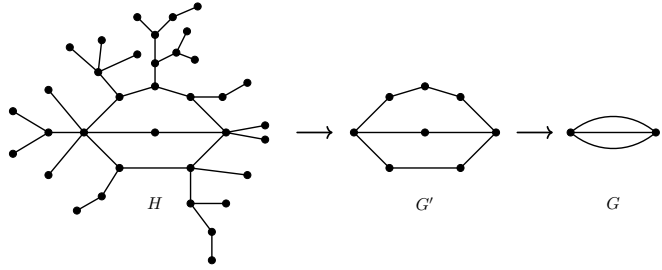
\begin{figure}[htb]
\centering
\resizebox{0.7\textwidth}{!}{
\begin{tikzpicture}[scale=0.8, line cap=round, line join=round]
\tikzset{
  v/.style={circle, fill=black, draw=black, inner sep=1.6pt},
  e/.style={line width=0.9pt, draw=black},
  arr/.style={->, line width=1.1pt, draw=black}
}

\begin{scope}[shift={(0,0)}]
  \coordinate (L)  at (-2,0);
  \coordinate (TL) at (-1,1);
  \coordinate (TC) at (0,1.3);
  \coordinate (TR) at ( 1,1);
  \coordinate (R)  at ( 2,0);
  \coordinate (BR) at ( 1,-1);
  \coordinate (BL) at (-1,-1);
  \coordinate (M)  at ( 0,0);

  \draw[e] (L)--(TL)--(TC)--(TR)--(R)--(BR)--(BL)--cycle;
  \draw[e] (L)--(M)--(R);

  \coordinate (L1)  at (-3,0);
  \coordinate (L1a) at (-4,0.6);
  \coordinate (L1b) at (-4,-0.6);
  \coordinate (L2)  at (-3,1.2);
  \coordinate (L3)  at (-3,-1.2);
  \draw[e] (L)--(L1) (L1)--(L1a) (L1)--(L1b);
  \draw[e] (L)--(L2) (L)--(L3);

  \coordinate (T1)  at (-1.6,1.7);
  \coordinate (T2)  at (-2.4,2.4);
  \coordinate (T3)  at (-1.5,2.6);
  \coordinate (T4)  at (-0.5,2.2);
  \draw[e] (TL)--(T1) (T1)--(T2) (T1)--(T3) (T1)--(T4) (T4);

  \coordinate (TC1) at (0.0,1.95);
  \coordinate (TC2) at (0.0,2.75);
  \coordinate (TC3) at (-0.5,3.25);
  \coordinate (TC4) at (0.5,3.25);
  \coordinate (TC5) at (1.2,3.55);
  \draw[e] (TC)--(TC1) (TC1)--(TC2) (TC2)--(TC3) (TC2)--(TC4);
 \draw[e] (TC4)--(TC5);

\coordinate (TC12) at (0.6,2.25);
\coordinate (TC13) at (0.9,2.85);
\coordinate (TC14) at (1.12,2.05);
  \draw[e] (TC1)--(TC12)--(TC13);
  \draw[e] (TC12)--(TC14);

  \coordinate (RT1) at (1.9,1.0);
  \coordinate (RT2) at (2.6,1.4);
  \draw[e] (TR)--(RT1)--(RT2);

  \coordinate (R1) at (3.1,0.2);
  \coordinate (R2) at (3.1,-0.2);
  \draw[e] (R)--(R1) (R)--(R2);

  \coordinate (B1) at (1,-2.0);
  \coordinate (B2) at (2,-2.0);
  \coordinate (B3) at (1.6,-2.8);
  \coordinate (B4) at (1.6,-3.6);
  \draw[e] (BR)--(B1)--(B2);
  \draw[e] (B1)--(B3)--(B4);

  \coordinate (BRL) at (2.6,-1.2);
  \draw[e] (BR)--(BRL);

  \coordinate (BL1) at (-1.5,-1.8);
  \coordinate (BL2) at (-2.2,-2.2);
  \draw[e] (BL)--(BL1)--(BL2);


  \foreach \p in {L,TL,TR,R,BR,BL,M} \node[v] at (\p) {};
  \foreach \p in {L1,L1a,L1b,L2,L3,T1,T2,T3,T4,TC1,TC2,TC3,TC4,RT1,RT2,R1,R2,B1,B2,B3,B4,BRL,BL1,BL2, TC5, TC, TC12, TC13, TC14}
    \node[v] at (\p) {};
        \node[scale=1.1] at (0,-2) {$H$};
\end{scope}

\draw[arr] (4,0) -- (5.0,0);

\begin{scope}[shift={(7.6,0)}]
  \coordinate (L)  at (-2,0);
  \coordinate (TL) at (-1,1);
  \coordinate (TC) at (0,1.3);
  \coordinate (TR) at ( 1,1);
  \coordinate (R)  at ( 2,0);
  \coordinate (BR) at ( 1,-1);
  \coordinate (BL) at (-1,-1);
  \coordinate (M)  at ( 0,0);

  \draw[e] (L)--(TL)--(TC)--(TR)--(R)--(BR)--(BL)--cycle;
  \draw[e] (L)--(M)--(R);

  \foreach \p in {L,TL,TR,R,BR,BL,M, TC} \node[v] at (\p) {};
      \node[scale=1.1] at (0,-2) {$G'$};
\end{scope}

\draw[arr] (10.2,0) -- (11.2,0);

\begin{scope}[shift={(12.9,0)}]
  \coordinate (A) at (-1.2,0);
  \coordinate (B) at ( 1.2,0);

  \draw[e] (A)--(B);                          
  \draw[e] (A) to[bend left=40]  (B);         
  \draw[e] (A) to[bend right=40] (B);         

  \node[v] at (A) {};
  \node[v] at (B) {};

    \node[scale=1.1] at (0,-2) {$G$};
\end{scope}

\end{tikzpicture}
}
\caption{A loop number $g=2$ example: reducing a graph $H\in \Gamma (2)$ to $G\in \Gamma_{\geq 3}$.}
\label{fig:reduction}
\end{figure}

Finally, the loop number $1$ case is exceptional. A connected graph $G$ with
loop number $1$ consists of a unique $k$-cycle $C_{k}$ and with rooted
trees attached to its vertices.  One can show that this leads to the
expression
\begin{equation}\label{eq:W1}
\sA_{1} (x) = \frac{1}{2}\sum_{k\geq 1} \frac{\sL^{k}}{k} =
\frac{1}{2}\log \Bigl(\frac{1}{1-\sL}\Bigr).
\end{equation}
Thus all the series $\sA_g$ can be built in a simple way from $\sL$.

\subsection{} Our goal in this section is to compute series akin to
Wright's but for the polynomials $W_{G} (t)$.  In particular we define
\[
\sW_{g} (x) = \sum_{G\in \Gamma (g)} \frac{W_{G} (t)
x^{v (G)} }{|\Aut G|} \in \QQ [t] \sets{x}.
\]
We follow the same strategy as Wright:
\begin{enumerate}
\item The first result is Theorem \ref{thm:loopnumber0}: we consider $g=0$ and define the analogue of the Lambert
series.  In fact we define three different series, depending on
whether the root is colored black or white plus some additional data.
Unlike the classical case, we cannot find explicit expressions for
these series.  However, we can easily evaluate them to any desired
precision using an iterative scheme.  Passing from this series to
$\sW_{0}$ is the same as \eqref{eq:W0}.
\item The next result is Theorem \ref{thm:loopnumber1}, the case $g=1$.  Again we build such graphs by
grafting trees onto a cycle, and the fact that the cycle graphs form a
rational family plays a key role.  In fact to carry this out we
slightly generalize the results of \S \ref{s:rationalfamilies} to
multivariate rational expressions.
\item Finally we consider the case $g\geq 2$ and prove Theorem \ref{thm:loopnumber2}.  Again we follow
Wright's method and show that the whole story boils down to
subdividing and grafting
trees onto graphs in the finite sets $\Gamma_{\geq 3} (g)$. 
\end{enumerate}

\subsection{}
We begin by defining the analogue of the rooted tree series $\sL$.  We define
\[
\sT (x) = \sum_{T} \frac{W_{T} (t)
x^{v (T)} }{|\Aut T|} \in \QQ [t] \sets{x},
\]
where the sum is taken over all rooted trees up to isomorphism.  Note
that here we only take the automorphisms of $T$ as a rooted tree,
meaning any automorphism must preserve the root. 

We will build $\sT$ by attaching sets of colored rooted trees to a new
root, just like in the functional equation \eqref{eq:lfunction}.
However to do this we need a refinement of $\sT$, since we must keep
track of the color of the new root, and the colors of the roots of the
old trees.  We define a decomposition
\begin{equation}\label{eq:tdecomp}
\sT = \sT_{w}^{+} + \sT_{w}^{-} + \sT_{b}
\end{equation}
where 
\begin{itemize}
\item $\sT_{w}^{+}$ consists of the contributions to $\sT$ of rooted
trees equipped with a coloring where the root is white and an
\emph{even} number of children are black; 
\item $\sT_{w}^{-} $ consists of the contributions to $\sT$ of rooted
trees equipped with a coloring where the root is white and an
\emph{odd} number of children are black; and
\item $\sT_{b}$ consists of rooted trees equipped with a coloring
where the root is black (there is no condition on children in this
case).
\end{itemize}

\subsection{}
Now we build a system of equations satisfied by the three generating
functions on the right of \eqref{eq:tdecomp}.  This is the
generalization of implicitly defining the Lambert series $\sL$ through
the relation \eqref{eq:lfunction}.  In fact, the expression
\eqref{eq:lfunction} is equivalent to the following combinatorial specification:
\emph{``Any rooted tree $(\sL )$ is obtained by attaching an arbitrary set
of rooted trees $(\exp \sL)$ to a new root $(x)$, where we join the
new root to the previous roots."}  We want to consider a similar
specification for the $W$-polynomials of trees in $\sT$.

We consider assembling a colored rooted tree by starting with a vertex
and joining it to a collection of rooted trees drawn from the three
types $\sT_{w}^{+}, \sT_{w}^{-}, \sT_{b}$.  The new root can either be
black or white, and how the polynomial $W$ changes, and which type we
build, depends on how we construct the new tree.

\begin{itemize}
\item To make a tree in $\sT_{w}^{+}$, we make the new root white,
attach trees arbitrarily from $\sT_{w}^{+}$ and $\sT_{w}^{-}$, but we
must attach an \emph{even} number of trees from $\sT_{b}$.  The
polynomial of the new tree gets multiplied by $t$ (since that is the
contribution of the new white root to the $W$-polynomial), and we
obtain the recursive relation
\[
\sT_{w}^{+} = txe^{\sT_{w}^{+}}e^{\sT_{w}^{-}}\cosh (\sT_{b}).
\]
Here we use $\cosh x = (e^{x}+e^{-x})/2$ since it is the even part of
the exponential function.
\item To make a tree in $\sT_{w}^{-}$, again we make the new root
white, attach trees arbitrarily from $\sT_{w}^{+}$ and $\sT_{w}^{-}$,
but we now must attach an \emph{odd} number of trees from $\sT_{b}$.
This time the polynomial of the new tree does not get multiplied by
$t$ (because the new root is now odd), and we obtain the recursive
relation
\[
\sT_{w}^{-} = xe^{\sT_{w}^{+}}e^{\sT_{w}^{-}}\sinh (\sT_{b}),
\]
where we use $\sinh x = (e^{x}-e^{-x})/2$ since it is the odd part of the exponential
function.
\item Finally to make a tree in $\sT_{b}$, we can join trees
arbitrarily from $\sT_{w}^{\pm}$, $\sT_{b}$ to a black root, but the
polynomial changes depending on how we select.  Each tree we pick from
$\sT_{w}^{+}$ (respectively, $\sT_{w}^{-}$) changes the polynomial by
$t^{-1}$ (resp.~$t$) since we change the parities of the neighborhoods
of the previous white roots.  This leads to the recursive relation
\[
\sT_{b} = x e^{t^{-1}\sT_{w}^{+} + t\sT_{w}^{-} + \sT_{b}}.
\]
\end{itemize}

Altogether we obtain a system of exponential generating functions:
\begin{equation*}
\left(\begin{array}{c}
\sT_{w}^{+}\\
\sT_{w}^{-}\\
\sT_{b}
\end{array} \right) = x \left(\begin{array}{c}
te^{\sT_{w}^{+}}e^{\sT_{w}^{-}}\cosh (\sT_{b})\\
e^{\sT_{w}^{+}}e^{\sT_{w}^{-}}\sinh (\sT_{b})\\
e^{t^{-1}\sT_{w}^{+} + t\sT_{w}^{-} + \sT_{b}}
\end{array} \right).
\end{equation*}
This can be solved iteratively, starting with the initial vector
\begin{equation}\label{eq:initialvector}
\left(\begin{array}{c}
tx\\
x^{2}\\
x
\end{array} \right)
\end{equation}
corresponding to the rooted colored trees in Figure \ref{fig:initialvector}.
We summarize these results in the following theorem:

\begin{theorem}\label{thm:loopnumber0}
Define a sequence of vectors $(a_{i},b_{i},c_{i})$, $i\geq 0$ by (i) $
(a_{0}, b_{0}, c_{0}) =  (tx,
x^{2}, x)$ and (ii) for $i>0$
\begin{itemize}
\item $a_{i} = t x \exp (a_{i-1} + b_{i-1})\cosh (c_{i-1})$,
\item $b_{i} = x \exp (a_{i-1} + b_{i-1})\sinh (c_{i-1})$, and 
\item $c_{i} = x \exp (t^{-1} a_{i-1} + t b_{i-1} + c_{i-1})$.
\end{itemize}
Then the sequence $(a_{i},b_{i},c_{i})$ converges to a vector of
formal power series $(a_{\infty}, b_{\infty}, c_{\infty})$ in $x$ with
coefficients rational polynomials in $t$.  We have
$\sT_{w}^{+}=a_{\infty}$, $\sT_{w}^{-} = b_\infty$, $\sT_{b} =
c_{\infty}$, and $\sT = a_{\infty}+b_{\infty}+c_{\infty}$.
\end{theorem}

\begin{figure}
\centering
\resizebox{0.5\textwidth}{!}{
\begin{tikzpicture}[font=\small, line width=0.7pt]

\def\R{0.9}      
\def\r{0.35}     
\def\dx{4.2}     
\def\labY{-1.8}  

\begin{scope}[shift={(0,0)}]
  \draw (0,0) circle (\R);
  \draw (0,0) circle (\r);
  \node[scale=1.8] at (0,\labY) {$a$};
\end{scope}

\begin{scope}[shift={(\dx,0)}]
  \draw (0,0) circle (\R);
  \draw (0,0) circle (\r);
  \draw (\r,0) -- (1.5,0);
  \fill (1.8,0) circle (0.32);
  \node[scale=1.8] at (0,\labY) {$b$};
\end{scope}

\begin{scope}[shift={(2*\dx,0)}]
  \draw (0,0) circle (\R);
  \fill (0,0) circle (\r);
  \node[scale=1.8] at (0,\labY) {$c$};
\end{scope}

\end{tikzpicture}
}
\caption{Trees giving the initial vector \eqref{eq:initialvector}.
The roots are circled.\label{fig:initialvector}}
\end{figure}
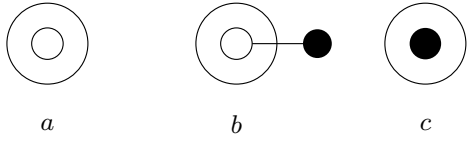

\subsection{}
Up to trees with $6$ vertices, the series $\sT$ is
\begin{multline*}
\sT  = (t + 1)x + (t^2 + 3)x^2 + (3/2t^3 + 9/2t + 6)x^3 +
(8/3t^4 + 8t^2 + 16t + 16)x^4 \\
+ (125/24t^5 + 205/12t^3 + 35t^2
+ 1465/24t + 145/3)x^5 \\
+ (54/5t^6 + 157/4t^4 + 80t^3 + 335/2t^2
+ 240t + 3073/20)x^6 + \dotsb .
\end{multline*}
For example, the coefficient of $x^{2}$ corresponds to the graph
$P_{2}$ with its two rootings.  The coefficient of $x^{4}$ corresponds
to two trees, the path graph $P_{4}$ and the star graph $S_{4}$ on $4$
vertices.  The former contributes $(t^{4}+2t^{2}+8t+5)$ for each of
its $2$ rootings up to isomorphism; there are no nontrivial
automorphisms in this case.  The latter has $W$-polynomial
$t^{4}+6t^{2}+9$.  There are two different rootings of $S_{4}$ giving
two different automorphism groups.  If we root at a leaf we have
$|\Aut S_{4}| = 2$, since we can permute the remaining two leaves.  If we
root at the center we have $|\Aut S_{4}| = 6$, since we can permute
all the leaves.  This means we get a contribution of $(t^{4}+6t^{2}+9)
(1/2+1/6)$, and altogether get $8/3t^4 + 8t^2 + 16t + 16$.

Once one has $\sT$, one can construct $\sW_0$ using the analogue of \eqref{eq:W0}:

\begin{corollary} 
We have 
$$
\sW_{0} (x) = \int x^{-1}\sT (x)\,dx.
$$
\end{corollary}

\subsection{}\label{ss:loopnumber1}
Next we consider the loop number $1$ case.  We will need a
generalization of the rational functions constructed in Theorem
\ref{thm:ratfamilies}.  In fact this arises as a simple generalization
of the $W$-polynomial.

\begin{definition}\label{def:fullW}
Let $G$ be a graph and let $c\in \sC (G)$ be a coloring.  We partition
the vertices $V (G)$ of $G$ as follows:
\begin{enumerate}
\item $B(c)$ is the subset colored black in $c$.
\item $W_{+} (c)$ is the subset colored white in $c$ with an
\emph{even} number of black neighbors.
\item  $W_{-} (c)$ is the subset colored white in $c$ with an
\emph{odd} number of black neighbors.
\end{enumerate}
Introduce variables $b, w_{+}, w_{-}$.  Then the \emph{full
$W$-polynomial} $\fullW_{G} \in \Z [b, w_{+}, w_{-}]$ is defined by 
\begin{equation}\label{eq:fullW}
\fullW_{G} = \sum_{c\in \sC (G)} \prod_{v\in B (c)} b \prod_{v\in
W_{+} (c)} w_{+} \prod_{W_{-} (c)} w_{-}. 
\end{equation}
\end{definition}

For example, $\fullW_{P_{2}} = b^{2}+2w_{-}b + w_{+}^{2}$.  Clearly 
$W_{G} (t)$ can be obtained from $\fullW_{G}$ through the
specialization $b\gets 1, w_- \gets 1, w_{+} \gets t$. 

We have the following generalization of Theorem \ref{thm:ratfamilies}:

\begin{theorem}\label{thm:rffull}
Suppose that $\sG$ is a rational family in Theorem \ref{thm:ratfamilies}.  Then the multivariate
generating function 
\[
\fullXi_{\sG} = \sum_{n\geq 1} \sum_{G\in \sG_{n}} \fullW_{G} (b,
w_{-}, w_{+})x^{n} \in \Z [b, w_{-}, w_{+}]\sets{x} 
\]
is a rational function in $x$.
\end{theorem}

\begin{proof}
The proof is a simple modification of the proof of Theorem
\ref{thm:ratfamilies}, so we leave the details to the reader.  The essential difference is
that now the monomial contributions $\mu$ in \eqref{eq:monomials}, which determine the entries
of the matrices $A$ and the vectors $v_{0}$, are
computed using the triple product in \eqref{eq:fullW} instead of the
single product in \eqref{eq:monomials}.
\end{proof}

For example, if $\sG$ is the cycle family $G_{n} = C_{n}$, the generating function $\fullXi $
is 
\begin{equation}\label{eq:fullxi}
\fullXi = \frac{(-2w_+^2 + 2w_-^2)bx^3 + 1}{(w_+^2 - w_-^2)bx^3 + (-b - w_+)x + 1}.
\end{equation}
Note that we have extended the cycle family to include $C_{1}$
(single vertex with a loop) and $C_{2}$ (two vertices, two parallel
edges), since we allow loops and parallel edge sets in this context.

\subsection{}
We are now ready to construct the loop number $1$ series $\sW_{1}$.
The construction is more complicated than that of $\sA_{1}$ since we
can't directly compose with the basic cycle generating function $\log (1/ (1-x))$.
Instead we proceed as follows.

We first take the series \eqref{eq:fullxi} and convolve with the
series 
\[
\frac{1}{2}\log \frac{1}{1-x} =  \frac{x}{2} + \frac{x^{2}}{4} + \frac{x^{3}}{6} + \dotsb  .
\]
This produces the series 
\begin{equation}\label{eq:s1}
 \sum_{n\geq 1} \fullW _{C_{n}} (b,w_{-},w_{+})\frac{ x^{n}}{2n}.
\end{equation}

Next we set $x\gets 1$ in \eqref{eq:s1} and regard it as a
multivariate formal power series in $\Q \sets{b, w_{-}, w_{+}}$.  Note
that we can do this since each polynomial $\fullW_{C_{n}}$ is
homogeneous of degree $n$.  Thus there are only finitely many terms of
bounded degree, and the expression makes sense as a multivariate
formal power series.  We denote the result by 
\begin{equation}\label{eq:s2}
\sZ = \sum_{n\geq 1} \frac{\fullW _{C_{n}} (b,w_{-},w_{+})}{|\Aut C_{n}|}.
\end{equation}
(This substitution would not be possible if we used the usual
$W$-polynomials and not the full $W$-polynomials.)

Now finally we graft rooted trees onto the cycles in \eqref{eq:s2},
i.e.~we plug in our generating functions $\sT_{b}$, $\sT_{w}^{-}$,
$\sT_{w}^{+}$ for the variables $b$, $w_{-}$, $w_{+}$.  Recall that
these are series in $x$ with coefficients polynomials in $t$ that
begin
\[
\sT_{b}= x + \dotsb , \quad \sT_{w}^{-} = x^{2} + \dotsb , \quad
\sT_{w}^{+} = tx + \dotsb .
\]
Thus after substitution we will end up with a power series in $x$ with
coefficients polynomials in $t$ that will be the desired
series $\sW_{1}$.  Note that this is the reason we get rid of all the $x$'s in
\eqref{eq:s2}: they are inserted back in by the tree series. We have to
carefully consider the effect of doing so on the parities of the white
vertices:
\begin{enumerate}
\item We can substitute $\sT_{b}$ in for the variable $b$, since $b$
marks a black vertex and $\sT_{b}$ has a black root.
\item We can replace $w_{+}$ with either $\sT_{w}^{+}$ or $\sT_{w}^{-}$.
No additional $t$ is needed for the root when substituting $\sT_{w}^{+}$ because
it is already included in $\sT_{w}^{+}$.
\item We can replace $w_{-}$ with either $\sT_{w}^{+}$ or
$\sT_{w}^{-}$.  This time in both cases we change the parity of the
root.  If we use $\sT_{w}^{+}$ then we must multiply by $1/t$, since
the even root is now odd after grafting.  If we use $\sT_{w}^{-}$ then
we must multiply by $t$, since the odd root is now even after
grafting.
\end{enumerate}

The final result is given by the following theorem:

\begin{theorem}\label{thm:loopnumber1}
In the series 
\[
\sZ = \sum_{n\geq 1} \frac{\fullW _{C_{n}} (b,w_{-},w_{+})}{|\Aut C_{n}|}.
\]
perform the following substitutions:
\[
b \gets \sT_{b}, \quad w_{+} \gets \sT_{w}^{+} + \sT_{w}^{-}, \quad w_{-} \gets t^{-1}\sT_{w}^{+} + t \sT_{w}^{-}.
\]
Then the result is the series $\sW_{1} (x)$.
\end{theorem}

The first few terms of $\sW_{1}$ are 
\begin{multline*}
\sW_{1} = (1/2t + 1/2)x + (3/4t^2 + 1/2t + 7/4)x^2 \\
+ (17/12t^3 +
1/2t^2 + 17/4t + 31/6)x^3 \\
+ (71/24t^4 + 3/4t^3 + 9t^2 + 73/4t +
131/8)x^4 + \cdots . 
\end{multline*}
For example, the coefficient of $x$ corresponds to the $1$-cycle of a
single vertex with a loop attached.  The $W$-polynomial is the same as
that of an isolated vertex, namely $t+1$, and we divide by $2$ to
account for the automorphism that flips the loop. The coefficient of
$x^{2}$ comes from two graphs: the $2$-cycle with $2$ vertices and $2$
parallel edges, which contributes $(t+1)^{2}/4$, and the $1$-cycle
with an edge attached, which contributes $(t^{2}+3)/2$.

\subsection{}\label{ss:loopnumber2}
Finally we come to the loop number $ g\geq 2 $ case.  The construction
of $\sW_{g}$ is very similar to that of $\sW_{1}$, so we will be brief.

We take the finite set of graphs $\Gamma_{\geq 3} (g)$.  For any $G\in
\Gamma_{\geq 3} (g)$, we consider the family $\sG = \sG (G) =
\{\sG_{n} \}_{n\geq 0}$ where $\sG_{n}$ is the set of all graphs
obtained by arbitrarily subdividing the edges of $G$ such that $n$ new
vertices have been added.  By \eqref{it:lastone} of Theorem
\ref{thm:ratfamilies}, for $G$ a simple graph the corresponding
(ordinary) generating function of $W$-polynomials is a rational
family.  It is easy to see that the same construction works when $G$
has loops and parallel edge sets, and when one incorporates
automorphisms to make the series exponential. Theorem \ref{thm:rffull}
then implies that the generating function of the full $W$-polynomials
is a rational function in $x$.  After the substitution $x\gets 1$ and
dividing by $|\Aut G|$ we obtain a multivariate generating function
\[
\sZ_{G} = \frac{1}{|\Aut G|}\sum_{n\geq 0} \sum_{H\in \sG_{n}}\fullW_{H} (b, w_{-}, w_{+}),
\]
which is well defined as a multivariate formal power series in $b, w_{-}, w_{+}$.

We can then obtain the contribution of $G$ to the series $\sW_{g}$ by
substituting into $\sZ_{G}$ the rooted tree series $\sT_{b},
\sT_{w}^{+}$, and $\sT_{w}^{-}$.  Similar considerations as in \S
\ref{ss:loopnumber1} apply, since we must be careful about the
destruction or creation of even white vertices.  The final result can
be stated as follows.

\begin{theorem}\label{thm:loopnumber2}
Let $g\geq 2$, and let $G\in \Gamma_{\geq 3} (g)$.  In the series 
\[
\sZ_{G} = \frac{1}{|\Aut G|}\sum_{n\geq 0} \sum_{H\in \sG_{n}}\fullW_{H} (b, w_{-}, w_{+})
\]
perform the substitutions
\[
b \gets \sT_{b}, \quad w_{+} \gets \sT_{w}^{+} + \sT_{w}^{-}, \quad w_{-} \gets t^{-1}\sT_{w}^{+} + t \sT_{w}^{-}
\]
and let $\sX_{G}(x) \in \QQ [t]\sets{x}$ be the result.
Then we have 
\[
\sW_{g} (x) = \sum_{G\in \Gamma_{\geq 3} (g)} \sX_{G} (x).
\]
\end{theorem}

We compute a few terms of the series $\sW_{2}
(x)$ to illustrate the result.  In this case $\Gamma_{\geq 3} (2)$
consists of the three graphs $G_1$, $G_2$, $G_3$ shown in Figure
\ref{fig:gamma}; the orders of their automorphism groups are indicated
under each graph.

Up to three vertices we have
\begin{multline}\label{eq:geq2series}
\sW_{2} (x) = (1/8t + 1/8)x + (7/12t^2 + 1/2t + 5/4)x^2 \\
+ (101/48t^3 + 9/8t^2 + 101/16t + 175/24)x^3 + \dotsb 
\end{multline}

\begin{figure}
\centering
\resizebox{0.75\textwidth}{!}{
\begin{tikzpicture}[line cap=round,line join=round,thick]

\begin{scope}[shift={(0,2.2)}]
  \coordinate (a) at (0,0);
  \coordinate (b) at (6,0);

  \draw (a) -- (b);
  \draw (a) .. controls (2,1.4) and (4,1.4) .. (b);
  \draw (a) .. controls (2,-1.4) and (4,-1.4) .. (b);

  \fill (a) circle (3pt);
  \fill (b) circle (3pt);

  \node[below=18mm] at ($(a)!0.5!(b)$) {$|\mathrm{Aut}\,G_1|=12$};
\end{scope}

\begin{scope}[shift={(10,2.2)}]
  \coordinate (c) at (0,0);

  \draw (c) .. controls (-2, 2.2) and (-2,-2.2) .. (c);
  \draw (c) .. controls ( 2, 2.2) and ( 2,-2.2) .. (c);

  \fill (c) circle (3pt);

  \node[below=-8mm] at (0,-2.6) {$|\mathrm{Aut}\,G_2|=8$};
\end{scope}

\begin{scope}[shift={(5,-1.6)}]
  \coordinate (u) at (0,0);
  \coordinate (v) at (4,0);

  \draw (u) -- (v);

  \draw (u) .. controls (-2, 2.2) and (-2,-2.2) .. (u);

  \draw (v) .. controls ( 6, 2.2) and ( 6,-2.2) .. (v);

  \fill (u) circle (3pt);
  \fill (v) circle (3pt);

  \node[below=8mm] at ($(u)!0.5!(v)$) {$|\mathrm{Aut}\,G_3|=8$};
\end{scope}
\end{tikzpicture}
}
\caption{The graphs in $\Gamma_{\geq 3}(2)$.}\label{fig:gamma}
\end{figure}
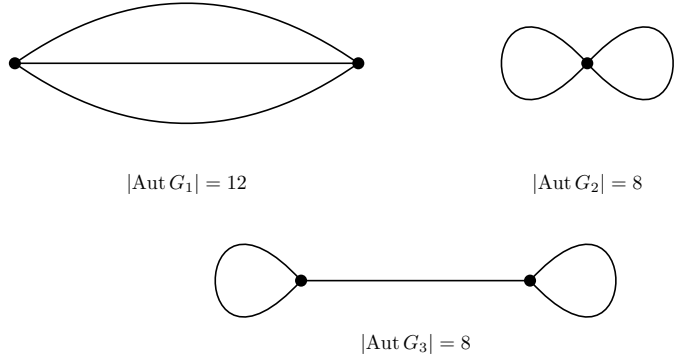

For example, consider the graph $G_2$.  This is the only $g=2$ graph
with one vertex, so it accounts for the $x$ term in
\eqref{eq:geq2series}.  The full $W$-polynomial is $b+w_+$, so we get
$(t+1)/8$. For the $x^{2}$ term, there are 4 graphs that contribute.
Two are $G_{1}$ and $G_{3}$; these contribute $ 1/12t^2 + 1/4$ and
$1/8t^2 + 3/8$ respectively.  The last two graphs come from
subdividing $G_{2}$ by adding a vertex, and by grafting a $P_{1}$ onto
the vertex of $G_{2}$.  These give respectively $1/4t^2 + 1/2t +
1/4$ and $1/8t^2 + 3/8$.  Adding these together gives $7/12t^2 + 1/2t + 5/4$.

\bibliographystyle{amsplain}
\bibliography{blackwhitegraphs}

\end{document}